\documentclass[final]{article}

\usepackage{showkeys}
\usepackage{amsmath,amsthm,amssymb,amsfonts}
\usepackage{graphicx}
\usepackage{color}

\newcommand{\R}{\mathbb{R}}

\newcommand{\un}{\mathbf{1}}
\newcommand{\eps}{\varepsilon}

\newcommand{\ha}{H^{\frac12} (\Omega) }
\newcommand{\had}{\dot{H}^{\frac12} (\Omega) }

\newcommand{\intom}{\int_\Omega}

\def\na{\nabla}
\def\pa{\partial}
\def\RR{\mathbb{R}}
\def\eps{\varepsilon}
\def\vphi{\varphi}
\def\lto{\longrightarrow}

\newtheorem{proposition}{Proposition}
\newtheorem{lemma}{Lemma}
\newtheorem{corollary}{Corollary}
\newtheorem{theorem}{Theorem}

\theoremstyle{definition}

\theoremstyle{remark}
\newtheorem{rem}{Remark}

\author{C. Imbert\footnote{CEREMADE, Universit\'e Paris-Dauphine, UMR
    CNRS 7534, place de Lattre de Tassigny, 75775 Paris cedex 16,
    France} and A. Mellet\footnote{Department of
    Mathematics. Mathematics Building. University of Maryland. College
    Park, MD 20742-4015, USAä}}

\title{Existence of solutions for a higher order non-local  equation appearing
in crack dynamics}

\begin{document}

\maketitle

\begin{abstract} In this paper, we prove the existence of non-negative
  solutions for a non-local higher order degenerate parabolic equation
  arising in the modeling of hydraulic fractures. The equation is
  similar to the well-known thin film equation, but the Laplace
  operator is replaced by a Dirichlet-to-Neumann operator,
  corresponding to the square root of the Laplace operator on a
  bounded domain with Neumann boundary conditions (which can also be
  defined using the periodic Hilbert transform).  In our study, we
  have to deal with the usual difficulty associated to higher order
  equations (e.g. lack of maximum principle).  However, there are
  important differences with, for instance, the thin film equation:
  First, our equation is nonlocal; Also the natural energy estimate is
  not as good as in the case of the thin film equation, and does not
  yields, for instance, boundedness and continuity of the solutions
  (our case is critical in dimension $1$ in that respect).
\end{abstract}



\section{Introduction}
This paper is devoted to the following problem:
\begin{equation}\label{eq:0}
\left\{
\begin{array}{ll}
u_t+(u^n I(u)_x)_x = 0 & \mbox{ for } x\in\Omega,\quad t>0\\[3pt]
u_x=0 ,\; u^n I(u)_x=0 & \mbox{ for } x\in\pa\Omega,\quad t>0\\[3pt]
u(x,0)=u_0(x)& \mbox{ for } x\in\Omega
\end{array}
\right.
\end{equation} 
where $\Omega$ is a bounded interval in $\RR$, $n$ is a positive real
number and $I$ is a non-local elliptic operator of order $1$
satisfying $I\circ I=-\Delta$; the operator $I$ will be defined
precisely in Section \ref{sec:pre} as the square root of the Laplace
operator with Neumann boundary conditions.  When $\Omega=\RR$, it
reduces to $I=-(-\Delta)^{1/2}$. In the sequel, we will always take
$\Omega=(0,1)$.

When $n=3$, this equation arises in the modeling of hydraulic
fractures. In that case, $u$ represents the opening of a rock fracture
which is propagated in an elastic material due to the pressure exerted
by a viscous fluid which fills the fracture (see Section
\ref{sec:phys} for details).  Such fractures occur naturally, for
instance in volcanic dikes where magma causes fracture propagation
below the surface of the earth, or can be deliberately propagated in
oil or gas reservoirs to increase production.  There is a significant
amount of work involving the mathematical modeling of hydraulic
fractures, which is beyond the scope of this article.  The model that
we consider in our paper, which corresponds to very simple fracture
geometry, was developed independently by Geertsma and De Klerk
\cite{DK} and Khristianovic and Zheltov \cite{ZK}.  Spence and Sharp
\cite{SS85} initiated the work on self-similar solutions and
asymptotic analyses of the behavior of the solutions of (\ref{eq:0}) near the tip of the fracture (i.e. the boundary of the support of $u$). There
is now an abundant literature that has extended this formal analysis
to various regimes (see for instance \cite{ad}, \cite{AP08},
\cite{MKP} and reference therein).  Several numerical methods have
also been developed for this model (see in particular Peirce et
al. \cite{PD08}, \cite{PD09}, \cite{PS05} and \cite{PS05b}).  However, to our
knowledge, there are no rigorous existence results for general initial
data. This paper is thus a first step toward a rigorous analysis of
(\ref{eq:0}).

\vspace{5pt}

From a mathematical point of view, the equation under consideration:
\begin{equation}\label{eq:main}
u_t+(u^n I(u)_x)_x = 0
\end{equation}
is a non-local parabolic degenerate equation of order $3$.  It is
closely related to the thin film equation which corresponds to the
case $I=\pa_{xx}$:
\begin{equation}\label{eq:tf}
u_t+(u^n u_{xxx})_x = 0
\end{equation}
(note that the porous media equation corresponds to the case
$I(u)=u$).  In particular, like the thin-film equation, Equation
(\ref{eq:main}) lacks a comparison principle, and the existence of a
non-negative solution (for non-negative initial data) is thus
non-trivial (it is well known that non-negative initial data may
generate changing sign solutions of the fourth order equation $\pa_t h
+ \pa_{xxxx} h=0$).

However, compared with (\ref{eq:tf}), the analysis of (\ref{eq:main})
presents some additional difficulties: First, the operator $I$ is
non-local and the algebra is not as simple as with the Laplace
operator. Second, because of the lower order of the operator $I$, the
natural regularity given by the energy inequality ($u\in H^{\frac12}$
rather than $u\in H^1$) does not give the boundedness and continuity
of weak solutions even in dimension $1$.
\medskip 

A remarkable feature of (\ref{eq:main}) and (\ref{eq:tf}) is
that the degeneracy of the diffusion coefficient permits the existence
of non-negative solutions.  In the case of the thin film equation
(\ref{eq:tf}), the existence of non-negative weak solutions was first
addressed by F. Bernis and A.~Friedman \cite{bf90} for $n> 1$.
Further results were later obtained, by similar technics, in
particular by E. Beretta, M. Bertsch and R. Dal Passo \cite{BBD} and
A. Bertozzi and M. Pugh \cite{BP1,BP2}.  Results in higher dimension
were obtained more recently (see \cite{Grun01,Grun95,DGG98}).
\vspace{10pt}

As in the case of the thin film equation, our approach to prove the
existence of solutions for (\ref{eq:0}) relies on a
regularization-stability argument, and the main tools are integral
inequalities which we present now.  \medskip

\paragraph{Integral inequalities.} Besides the conservation of mass,
the solutions of (\ref{eq:main}) satisfy two important inequalities
(that have a counterpart for the thin film equation): Assuming that we
have $\Omega=\R$ and $I=-(-\Delta)^{1/2}$ (for the time being), we can
indeed easily show that the solutions of (\ref{eq:main}) satisfy the
energy inequality
\begin{equation}\label{eq:energy} - \int_\Omega u(t) I(u(t))\, dx
  +\int_0^T\int_\Omega u^n (I(u)_x)^2\, dx\, dt \leq - \int_\Omega u_0
  I(u_0)\, dx
\end{equation}
(where $-\int u I(u)\, dx$ is the homogeneous $H^{1/2}$ norm)
and an entropy like inequality
\begin{equation}\label{eq:entropy} 
  \int_\Omega G(u(t))\, dx -\int_0^T\int_\Omega u_x I(u)_x\, dx\, dt =\int_\Omega G(u_0)\, dx
\end{equation}
where $G''(s)=\frac{1}{s^n}$. 
\medskip

The energy inequality (\ref{eq:energy}) controls the
$L^\infty(0,T;H^{1/2}(\Omega))$ norm of the solutions (instead of
$L^\infty(0,T;H^{1}(\Omega))$ for the thin film equation).  We see
here that the order $1/2 $ for the operator $I$ is critical in
dimension $1$ in the sense that we are just short of a
$L^\infty((0,T)\times \Omega)$ estimate and continuity of the
solutions.  Because of that fact, many of the arguments used in the
analysis of the thin film equation do not apply directly to our case.

Next, we observe that as in the case of the thin film equation, the
entropy inequality (\ref{eq:entropy}) provides some control on some
negative power of $u$ for $n> 2$.  Indeed, we can take
$$ G(s) = \int_1^s \int_1^r \frac{1}{t^n}\, dt\, dr$$
so that $G$ is a nonnegative convex function satisfying $G'(1)=0$ and
$G(1)=0$.  This yields:
\begin{equation}\label{eq:G}
G(s)=\left\{
\begin{array}{ll}
\displaystyle s \ln s -s +1 & \mbox{ when } n=1 \\[10pt]
\displaystyle -\frac{s^{2-n}}{(2-n)(n-1)} +\frac{s}{n-1}+ \frac{1}{2-n} & \mbox{ when } 1< n< 2 \\[10pt]
\displaystyle \ln \frac 1 s +s-1 & \mbox{ when } n=2\\[8pt]
\displaystyle \frac{1}{(n-2)(n-1)}\frac{1}{s^{n-2}} +\frac{s}{n-1} -\frac{1}{n-2} & \mbox{ when } n> 2.
\end{array}
\right.
\end{equation}
Note that $G(0)=+\infty$ when $n\geq 2$, while $G$ is bounded in a
neighborhood of $0$ for $n\in[1,2)$.  This will be key in proving the
non-negativity of the solution.  The entropy equality also gives some
control on the $L^2(0,T;H^{3/2}(\Omega))$ norm of the solution which
will be crucial in getting the necessary compactness in the construction of the  solution. 
However, in order to make use of this
inequality, we need $\int_\Omega G(u_0)\, dx$ to be finite, which,
when $n\geq 2$ prohibits compactly supported initial data.  \medskip

Besides those two inequalities, there are several other integral
estimates that have proved extremely useful in the study of the thin
film equation.  The simplest one are local versions of
(\ref{eq:energy}) and (\ref{eq:entropy}). However, because of the
nonlocal character of the operator $I$, it seems very delicate to
establish similar inequalities for (\ref{eq:main}).

Another crucial estimate in the analysis of \eqref{eq:tf}, established by
Bernis \cite{bernis96c}, is the following:
$$\int (u^{\frac{n+2}{2}}_{xxx})^2\, dx \leq C \int u^n u_{xxx}^2\, dx$$
for $n \in (\frac12,3)$.  Such an inequality yields important
estimates from the dissipation in the energy inequality (despite the
degeneracy of the diffusion coefficient).  Again it is not clear what
would play the role of this inequality in our situation.  The same
remark holds for the so called $\alpha$-entropy
\cite{bf90,BBD,BP2,BDGG98}: for $\alpha\in
(\max(-1,\frac{1}{2}-n),2-n)$, $ \alpha \neq 0$, it can be proved that
the solutions of the thin film equation (\ref{eq:tf}) satisfy:
\begin{multline*}
  \frac{1}{\alpha+1} \int u^{\alpha+1} (\cdot,T)\, dx + C \int_0^T\int
  \big(|\pa_x u^{\frac{\alpha+n+1}{4}}|^4+|\pa_{xx}u^{\frac{\alpha+n+1}{2}}|^2\big) \, dx\, dt \\
  \leq \frac{1}{\alpha+1} \int u^{\alpha+1} _0\, dx.
\end{multline*}

These last two inequalities are essential in establishing many
qualitative properties of the solutions, such as finite speed
expansion of the support and waiting time phenomenon.  Though we
expect such properties to hold for (\ref{eq:main}) as well, it is not
clear at this point how to deal with the non local character of $I$.
\medskip

Finally, let us comment on the power of the diffusion coefficient
$u^n$.  Interestingly, the power $n=3$, which is the physically
relevant power in our model, is critical in many results for the thin
film equation. In particular many existence and regularity results (as
well as waiting time results) are only valid for $n\in(0,3)$.  It is
actually believed that for $n\geq 3$, (and it is proved for $n\geq 4$)
the support of the solutions of (\ref{eq:tf}) does not expand.  It is
not clear what would be the critical exponent for (\ref{eq:main}),
though numerical results suggest that for $n=3$, the support of the
solutions of (\ref{eq:main}) does expand for all time.

\medskip

\paragraph{Main results.}
A weak formulation of (\ref{eq:main}) is given by
$$   
\iint_Q u \; \partial_t \vphi\, dx\, dt + \iint_Q u^n \pa_xI (u) \;
\pa_x \vphi\, dx\,dt = - \int_\Omega u_0 \vphi(0,\cdot)\, dx
$$
for all $\vphi \in \mathcal{D} (\overline{Q})$ where $Q$ denotes
$\Omega \times (0,T)$.  However, because of the degeneracy of the
coefficient $u^n$, it is difficult to give a meaning to the term $
u^n\pa_xI(u)$. We thus perform an additional integration by parts to
get the following weak formulation of (\ref{eq:main}):
\begin{eqnarray}
  \iint_Q u \; \partial_t \vphi\, dx\, dt - \iint_Q n u^{n-1}\pa_x u \, I (u) \; \pa_x \vphi\, dx\,dt  
  - \iint_Q u^n  \, I (u) \; \pa_{xx} \vphi\, dx\,dt  \nonumber \\
  = - \int_\Omega u_0 \vphi(0,\cdot)\, dx\label{eq:weak}
\end{eqnarray}
for all $\vphi \in \mathcal{D} (\overline{Q})$ satisfying $\pa_x \vphi|_{\pa\Omega}=0$.

We are going to prove the following existence theorem:
% -------------------------------------------------------------------------------------
\begin{theorem}\label{thm:main}
Assume $n\geq 1$.
For any non-negative initial condition $u_0 \in \ha$ such that
\begin{equation}\label{eq:initpositive} 
\int_\Omega G(u_0)\, dx <\infty,
\end{equation}
there exists a non-negative function $u \in L^\infty (0,T,\ha)$ such
that
\begin{equation}\label{estim:nonlinear}
  u\in L^2(0,T,H^{\frac32}_N(\Omega))
\end{equation}
which satisfies \eqref{eq:weak}
 for all $\vphi \in \mathcal{D} (\overline{Q})$ satisfying $\pa_x \vphi|_{\pa\Omega}=0$. 
\item Furthermore  $u$ satisfies, for almost every $t \in (0,T)$
\begin{eqnarray}
\label{estim:conservation mass}
\intom u (t,\cdot ) \, dx= \intom u_0 \, dx , \\
\label{estim:nrj}
\|u (t,\cdot)\|_{\ha}^2 + 2  \int_0^t \int_\Omega g^2\, dx\, ds  \le \|
u_0 \|_{\ha}^2 ,
\end{eqnarray}
where the function $g \in L^2 (Q)$ satisfies 
$g=\pa_x (u^{\frac n 2} I(u))-\frac n 2 u^{\frac {n-2} 2}\pa_x u\; I(u)$ in
$\mathcal D'(\Omega)$, and 
\begin{equation}
\label{estim:entropy}
\intom G(u(x,t)) \, dx+||u||^2_{L^2(0,t;\dot{H}^{\frac32}_N(\Omega))} 
\le \intom G(u_0(x))\, dx  .
\end{equation}
\end{theorem}

We recall that the function $G:(0,\infty)\to \R_+$ is given by (\ref{eq:G}). The space $H^{\frac32}_N(\Omega)$ appearing in (\ref{estim:nonlinear}) will be defined precisely in Section
\ref{sec:pre}. In particular, the following characterization will be
given:
$$ 
H^{\frac32}_N(\Omega) = \left\{u\in H^{\frac32}(\Omega)\, ;\, \int_\Omega
  \frac{u_x^2}{d(x)}\, dx<\infty \right\}
$$
where $d(x)$ denotes the distance to $\pa \Omega$.  Condition
(\ref{estim:nonlinear}) thus implies that $u$ satisfies $u_x=0$ on
$\partial \Omega$ in some weak sense.  \vspace{5pt}

Note that at least formally, we have $g=u^{\frac n 2} \pa_x I(u)$
  (though we do not have enough regularity on $u$ to give a meaning to
  this product in general).  
Finally, we point out that we have $H^{\frac32}_N (\Omega) \subset W^{1,p}(\Omega)$ for all
$p<\infty$ and so every terms in \eqref{eq:weak} makes sense.
\vspace{10pt}

For $n\geq 2$, condition (\ref{eq:initpositive}) requires in
particular that $\mathrm{Supp}(u_0)=\Omega$ and inequality
(\ref{estim:entropy}) implies that this remains true for all positive
time.  This is a serious restriction since the case of compactly
supported initial data is physically the most interesting (see Section
\ref{sec:phys}). We hope to be able to get rid of this assumption in a
further work.

For $n >3 $, we can actually show that condition
(\ref{eq:initpositive}) requires $u(\cdot ,t)$ to be strictly positive
for a.e. $t\in (0,T)$.  In fact, we can prove:
\begin{theorem}\label{thm:2}
  When $n>3$, there exists a set $P\subset (0,T)$ such that
  $|(0,T)\setminus P|=0$ and the solution $u$ given by Theorem
  \ref{thm:main} satisfies
$$
u(\cdot, t)\in \mathcal C^\alpha(\Omega) \mbox{ for all $t\in P$ and for all $\alpha<1$}, 
$$ 
and $u(\cdot,t)$ is strictly positive in $\Omega$.  Finally, $u$
solves
$$ 
\pa_t u +\pa_x J =0 \qquad \mbox{ in } \mathcal D'(\Omega),
$$
where
$$ 
J(\cdot ,t)= u^n \pa_x I(u)\in L^1(\Omega) \quad \mbox{ for all } t\in P.
$$
\end{theorem}
\medskip

\paragraph{Organization of the article.}
The paper is organized as follows: In Section~\ref{sec:phys}, we give a brief description 
of the mathematical  modeling of hydraulic fracture which gives rise to equation (\ref{eq:main}) with $n=3$. 
In Section~\ref{sec:pre}, we introduce the functional analysis tools that
will be needed to prove Theorem \ref{thm:main}. In particular, the
non-local operator $I(u)$ is defined, first using a spectral
decomposition, then as a Dirichlet-to-Neuman map. An integral
representation for $I$, using the periodic Hilbert transform is also
given.  Section~\ref{sec:regularized} is devoted to the study of a
regularized equation while the proof of Theorem \ref{thm:main} is
given in Sections~\ref{sec:main-proof} (for the case $n\geq 2$) and
\ref{sec:other-proof} (for the case $n\in[1,2)$).  Theorem \ref{thm:2}
is proved in Section~\ref{sec:thm2-proof}.

\paragraph{Acknowledgements.} The authors would like to thank
A. Pierce for bringing this model to their attention and for many very
fruitful discussions during the preparation of this article.  The
first author was partially supported by the ANR-projects ``EVOL'' and
``MICA''.  The second author was partially supported by NSF Grant
DMS-0901340.

\section{The physical model}\label{sec:phys}
 When $n=3$, Equation (\ref{eq:main}) can be used to model the propagation of an
impermeable KGD fracture driven by a viscous fluid in a uniform
elastic medium under condition of plane strain.  More precisely,
denoting by $(x,y,z)$ the standard coordinates in $\R^3$, we consider
a fracture which is invariant with respect to one variable (say $z$)
and symmetric with respect to another direction (say $y$).  The
fracture can then be entirely described by its opening $u(x,t)$ in the
$y$ direction (see Figure \ref{fig:1}).
\begin{figure}\label{fig:1} 
\begin{center}
\includegraphics[height=5cm]{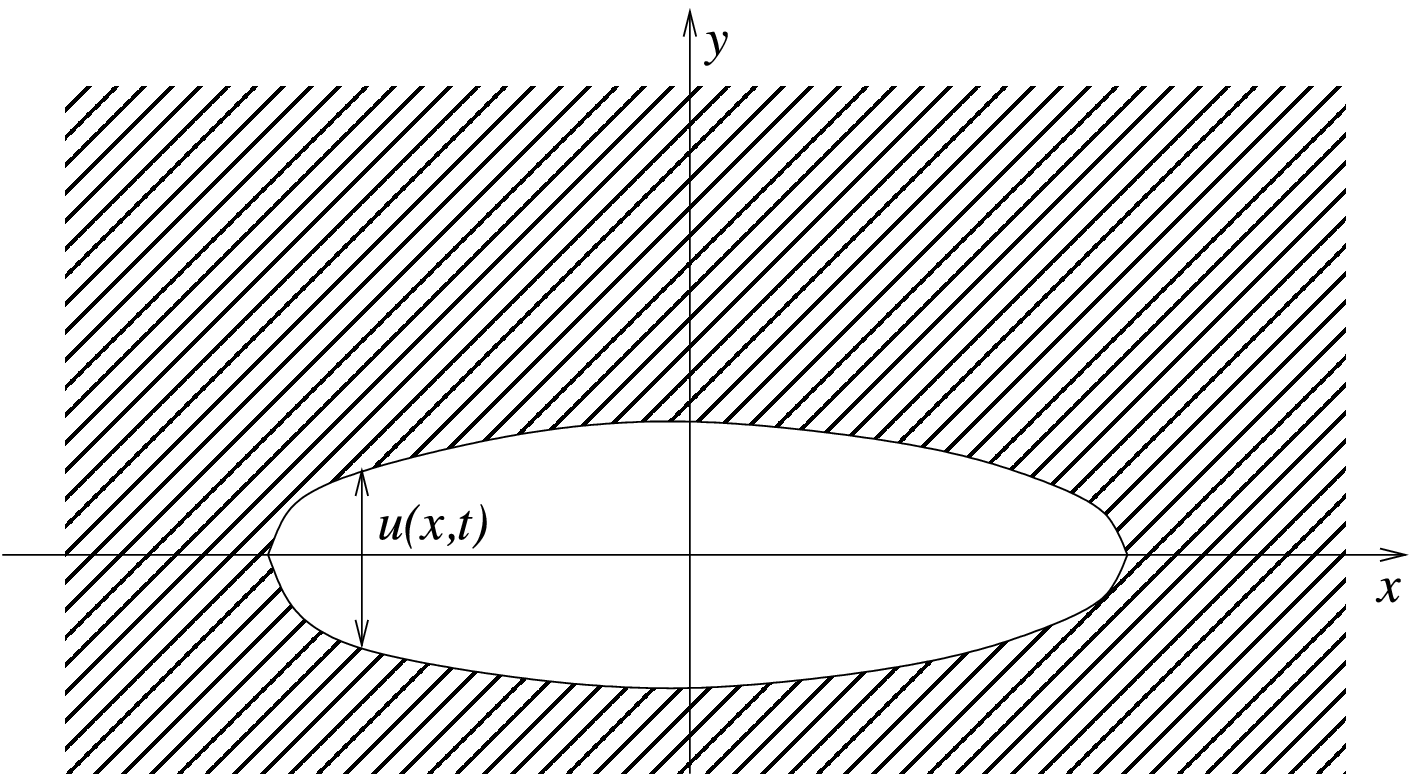}
\end{center}
\end{figure}
Since it assumes that the fracture is an infinite strip whose
cross-sections are in a state of plane strain, this model is only applicable
to rectangular planar fracture with large aspect ratio.

We now briefly describe the main steps of the derivation of (\ref{eq:0}).

\subsection{Conservation of mass and Poiseuille law}

The conservation of mass for the fluid inside the fracture, averaged
with respect to $y$ yields:
\begin{equation}\label{eq:conservation}
 \partial_t (\rho u) + \partial_x q = 0 \quad \text{ in } \R
\end{equation}
where $\rho$ is the density of the fluid (which is assumed to be constant)
and $q=q (x,t)$ denotes the fluid flux.
This flux is given by 
\begin{equation}\label{eq:flux}
q= \rho u \overline v
\end{equation}
where $\overline v$ is the $y$-averaged horizontal component of the
velocity of the fluid
$$ \overline v = \frac{1}{u}\int_{-u/2}^{u/2} v_H(t,x,y)\, dy.$$
Under the lubrification approximation, Navier-Stokes equations reduce
to
$$ 
\mu \frac{\partial^2 v_H}{\partial y^2}(t,x,y) = \partial_x p(x,t)
$$
where $p$ denotes the pressure of the fluid at a point $x$ and $\mu$
denotes the viscosity coefficient.  Assuming a no-slip boundary
condition $v=0$ at $y=\pm u/2$, we deduce
$$ 
v_H (t,x,y)=\frac{1}{\mu} \partial_x p
\left[\frac{1}{2}y^2-\frac{1}{8}u^2\right] \quad \mbox{ for $-u\leq y
  \leq u$}
$$
and so
$$
\overline v (x,t) = - \frac{u^2}{12 \mu} \partial_x p(x,t)
$$
Using (\ref{eq:flux}), we deduce {\it Poiseuille law}
\begin{equation}\label{loi d'etat}
q  = - \frac{u^3}{12 \mu} \partial_x p\, .
\end{equation}
Together with (\ref{eq:conservation}), this implies
$$  \pa_t u -\pa_x\left(\frac{u^3}{12\mu}\pa_x p\right)=0.$$
In order to obtain (\ref{eq:main}), it only remains to express the
pressure $p$ as a function of the displacement $u$ (i.e. $p=-I(u)$).

\subsection{The pressure law}
For a state of plane strain, the elasticity equation expresses the
pressure as a function of the fracture opening. After a rather
involved computation, one can derive the following nonlocal expression
(see \cite{CS83}):
$$p(x,t)=-\frac{E'}{4\pi} \int_\R \frac{\pa_x u(z,t)}{z-x}\, dz$$
where $E'$ denotes Young's modulus.  Denoting by $\mathcal H$ the
Hilbert transform, we can rewrite this formula as
$$p(x,t)=\frac{E'}{4}\mathcal H(\pa_x u) = \frac{E'}{4} (-\Delta)^{1/2} u (x,t)$$
where $(-\Delta)^{1/2}$ is the half-Laplace operator, defined, for
instance, using the Fourier transform by $ \mathcal F((-\Delta)^{1/2}
u)=|\xi| \mathcal F(u)$.  \medskip

It is well known that the half Laplace operator can also be defined as
a Dirichlet-to-Neumann map for the harmonic extension. More precisely,
the pressure is given by
\begin{equation}\label{p as a function of displacement}
p (x) = \frac{E'}{4} \partial_y v (x,0) 
\end{equation}
where $v$ solves
\begin{equation}\label{eq elasticity}
\left\{ \begin{array}{ll}
- \Delta v  = 0 & \text{ in } \R \times (0,+\infty), \\
v (x,0) = u (x,t), & \text{ on } \R \, . 
\end{array} \right.
\end{equation}
By taking advantage of the symmetry
of the problem, the function $v(x,y)$ can be interpreted as the displacement 
of the rock. 
Denoting $I(u)=-(-\Delta)^{1/2}u$, 
we deduce $ p= -\frac{E'}{4} I(u)$ and so
$$  \pa_t u +\frac{E'}{48\mu} \pa_x\left(u^3\pa_x I(u)\right)=0.$$

\paragraph{A technical assumption.} In order to reduce the technicality
of the analysis, we will assume that the crack is periodic with
respect to $x$. Since we expect compactly supported initial data to
give rise to compactly supported solutions whose supports expand with
finite speed, this is a reasonnable physical assumption. We also
assume that the initial crack is even with respect to $x=0$ and we
look for solutions that are also even.

By making such assumptions (periodicity and evenness), we can replace
\eqref{eq elasticity} with the following boundary value problem
$$
\left\{
\begin{array}{ll}
-\Delta v = 0 & \mbox{ in }  \Omega \times  (0,\infty) \\
v_\nu = 0 & \mbox{ on } \pa\Omega \times(0,\infty) \\
v(x,0)= u(x) &\mbox{ on } \Omega 
\end{array}
\right.
$$
with $\Omega=(0,1)$ if the period of the initial crack is $2$. The cylinder 
$\Omega \times (0,+\infty)$ is denoted by $C$ in the remaining
of the paper.

\paragraph{Mathematical definition of the pressure.}
It turns out that it is easier to define first the operator $I$ by
using the spectral decomposition of the Laplace operator: We take
$\{\lambda_k,\vphi_k\}$ the eigenvalues and corresponding eigenvectors
of the Laplacian operator in $\Omega$ with Neumann boundary conditions
on $\pa\Omega$:
$$
\left\{
\begin{array}{ll}
-\Delta \vphi_k = \lambda_k\vphi_k & \mbox{ in } \Omega\\
\pa_\nu \vphi_k = 0 & \mbox{ on } \pa\Omega.
\end{array}
\right. 
$$
We then define the operator $I$ by
$$
I(u):=\sum_{k=0}^{\infty} c_k \vphi_k(x) \mapsto - \sum_{k=0}^{\infty}
c_k \lambda_k^{\frac12} \vphi_k(x)
$$
which maps $H^1(\Omega)$ onto $L^2(\Omega)$. We will prove that this
operator can be characterized as a Dirichlet-to-Neumann map (see
Proposition~\ref{prop:dir-to-neum}) and that it  has an integral representation
as well (see Proposition~\ref{prop:integral-rep}).

\subsection{Boundary conditions}

The opening $u$ is solution of (\ref{eq:main}) in its
support. The model must thus be supplemented with some {\it boundary
  conditions} at the tip of the fracture.
 
Assuming that $\mathrm{Supp} (u(t,\cdot))=[-\ell(t),\ell(t)]$, it is 
usually assumed that
$$ 
u=0,\qquad u^3 \pa_x p = 0 \qquad \mbox{ at } x=\pm \ell (t)
$$
which ensures zero width and zero fluid loss at the tip.

We point out that the model is a free boundary problem, since the
support $[-\ell(t),\ell(t)]$ of the fracture is not known a priori.
Since the equation is of order $3$, those two conditions are not
enough to fully determine a solution.  In fact, there should be an
additional condition which takes into account the energy required to
break the rock at the tip of the crack.  Consistent with linear
elastic fracture propagation, we can assume that the rock toughness
$K_{Ic}$ equals the stress intensity factor $K_I$.  When the crack
propagation is determined by the toughness of the rock, a formal
asymptotic analysis of fracture profile at the tip (see
\cite{ad,ddlp}) then shows that
\begin{equation}\label{eq:fb} 
u(x,t) \sim \frac{K'}{E'} \sqrt{\ell(t)-x}  \qquad \mbox{ as } x\rightarrow \ell(t)
\end{equation}
with $K'=4\sqrt{\frac{2}{\pi}} K_{Ic}$ (and a similar condition for
$x\rightarrow -\ell$).  One can now take this condition on the profile
of $u$ at the tip as the missing free boundary condition.  The
resulting free boundary problem is clearly very delicate to study
(remember that (\ref{eq:0}) is a third order degenerate non-local
parabolic equation).
\medskip

A particular case which is simpler and still interesting is the case of zero
toughness ($K_{Ic}=0$).  This is relevant mainly if there is a
pre-fracture (i.e. the rock is already cracked, even though $u=0$
outside the initial fracture).  Mathematically speaking, this means
that Equation~(\ref{eq:0}) is now satisfied everywhere in $\R$ even
though $u$ is expected to have compact support. No free boundary
conditions are necessary.  This is the problem that we are considering
in this paper.

Note that one should then have $\lim_{x\rightarrow \ell}
(\ell(t)-x)^{-1/2} u(x,t)=0$ at the tip of the crack.  In fact, formal
arguments show that the asymptotic behavior of the fracture opening
near the fracture tip should be proportional to $ (\ell(t)-x)^{2/3}$
(see \cite{ad,ddlp}).

This approach is very similar to what is usually done with the porous
media equation, and it has been used successfully in the case of the
thin film equation to prove the existence of solutions with zero
contact angle (in that case, we speak of precursor film, or
pre-wetting). The study of the free boundary problem with free
boundary condition (\ref{eq:fb}) would be considerably more difficult
(one would expect the gradient flow approach developed by F. Otto
\cite{Otto98} for the thin film equation with non zero contact angle
to yield some result when $n=1$).

\section{Preliminaries}\label{sec:pre}

In this section, we define the operator $I$ and give the functional
analysis results that will play an important role in the proof of the
main theorem.  A very similar operator, with Dirichlet boundary
conditions rather than Neumann boundary conditions, was studied by
Cabr\'e and Tan \cite{ct09}.  This section follows their analysis very
closely.

\subsection{Functional spaces}

\paragraph{The space $H^s_N (\Omega)$.}
We denote by $\{\lambda_k,\vphi_k\}_{k= 0,1,2\dots}$ the eigenvalues
and corresponding eigenfunctions of the Laplace operator in $\Omega$
with Neumann boundary conditions on $\pa\Omega$:
\begin{equation}\label{eq:eigen}
\left\{
\begin{array}{ll}
-\Delta \vphi_k = \lambda_k\vphi_k & \mbox{ in } \Omega\\
\pa_\nu \vphi_k = 0 & \mbox{ on } \pa\Omega,
\end{array}
\right. 
\end{equation}
normalized so that $\int_\Omega \vphi_k^2\, dx=1$.  When $\Omega =
(0,1)$, we have
$$
\lambda_0=0\, , \qquad \vphi_0(x) = 1
$$
and
$$ 
\lambda_k = (k\pi)^2 \, , \qquad \vphi_k(x) = \sqrt 2 \cos(k\pi x)\,
\qquad k= 1,\,2,\, 3, \, \dots
$$
The $\vphi_k$'s clearly form an orthonormal basis of $L^2(\Omega)$.
Furthermore, the $\vphi_k$'s also form an orthogonal basis of the
space $H^s_N(\Omega)$ defined by
$$
H^s_N(\Omega) = \left\{u=\sum_{k=0}^\infty c_k\vphi_k \, ;\,
  \sum_{k=0}^\infty c_k^2(1+\lambda_k^s) <+\infty \right\}
$$
equipped with the norm
$$ 
||u||_{H^s_N(\Omega)}^2= \sum_{k=0}^\infty c_k^2(1+\lambda_k^s) 
$$
or equivalently (noting that $c_0 = \|u\|_{L^1(\Omega)}$ and
$\lambda_k\geq 1$ for $k\geq 1$):
$$ 
||u||_{H^s_N(\Omega)}^2= \|u\|_{L^1 (\Omega)}^2 + \| u\|^2_{\dot{H}^s_N (\Omega)}
$$
where the homogeneous norm is given by:
$$
\| u\|^2_{\dot{H}^s_N (\Omega)} = \sum_{k=1}^\infty c_k^2\lambda_k^s .
$$

\paragraph{A characterisation of $H^s_N (\Omega)$.}
The precise description of the space $H^s_N(\Omega)$ is a classical
problem.

Intuitively, for $s<3/2$, the boundary condition $u_\nu=0$ does not
make sense, and one can show that (see Agranovich and Amosov
\cite{aa03} and references therein):
$$ 
H^s_N(\Omega)=H^s(\Omega) \quad \mbox{ for all } 0\leq s<\frac32.
$$
In particular, we have $H^\frac{1}{2}_N(\Omega)=\ha$ and we will see
later that 
$$
\|u\|^2_{\dot{H}^{\frac12} (\Omega)} = \int_\Omega \int_\Omega (u(y)-u(x))^2 \nu (x,y) dx dy
$$
where $\nu (x,y)$ is a given positive function; see \eqref{defi:nu} below.

For $s>3/2$, the Neumann condition has to be taken into account, and
we have in particular
$$  
H^2_N(\Omega) = \{ u\in H^2(\Omega)\, ;\, u_\nu=0 \mbox{ on }\pa\Omega\}
$$
which will play a particular role in the sequel.  More generally, a similar
characterization holds  for $3/2<s<7/2$. For $s>7/2$, additional
boundary conditions have to be taken into account.

The case $s=3/2$ is critical (note that $u_\nu|_{\pa\Omega}$ is not
well defined in that space) and one can show that
\begin{eqnarray*} 
  H^{\frac32}_N(\Omega) 
  & = & \left\{ u\in H^{\frac32}(\Omega)\, ;\, \int_\Omega \frac{u_x^2}{d(x)}\, dx <\infty\right\}  
\end{eqnarray*}
where $d(x)$ denotes the distance to $\pa\Omega$. A similar result
appears in \cite{ct09}; more precisely, such a characterization of
$H^{\frac32}_N(\Omega)$ can be obtained by considering functions $u$
such that $u_x \in \mathcal V_0(\Omega)$ where $\mathcal V_0(\Omega)$
is defined in \cite{ct09} as the equivalent of our space
$H^{1/2}_N(\Omega)$ with Dirichlet rather than Neumann boundary
conditions.  We do not dwell on this issue since we will not need this
result in our proof.

\subsection{The operator $I$}

As it is explained in the Introduction, the operator $I$ is related to
the computation of the pressure as a function of the
displacement. From this point of view, the operator $I$ is a
Dirichlet-to-Neumann operator associated with the Laplacian. Since we
study the problem in a periodic setting we explained that this yields
to consider Neumann boundary conditions on a cylinder $C=\Omega\times
(0,+\infty)$.

\paragraph{Spectral definition.} It is convenient to begin with the
spectral definition of the operator $I$: With $\lambda_k$ and
$\vphi_k$ defined by (\ref{eq:eigen}), we define the operator
\begin{equation}\label{eq:Idef}
  I: \sum_{k=0}^{\infty} c_k \vphi_k \; \longmapsto \; 
- \sum_{k=0}^{\infty} c_k \lambda_k^{\frac12} \vphi_k
\end{equation}
which clearly maps $H^1(\Omega)$ onto $L^2(\Omega)$ and
$H^2_N(\Omega)$ onto $H^1(\Omega)$.

\paragraph{Dirichlet-to-Neuman map.} 
With the spectral definition in hand, we are now going to show that
$I$ can also be defined as the Dirichlet-to-Neumann operator
associated with the Laplace operator supplemented with Neumann
boundary conditions.

To be more precise, we consider the following boundary problem in the
cylinder $C=\Omega \times (0,+\infty)$:
\begin{equation}\label{eq:exp}
\left\{
\begin{array}{ll}
-\Delta v = 0 & \mbox{ in } C, \\
v(x,0)= u(x) &\mbox{ on } \Omega, \\
v_\nu = 0 & \mbox{ on } \pa \Omega \times (0,\infty).
\end{array}
\right.
\end{equation}
We will show that we have
$$ I(u)=\pa_y v(\cdot,0).$$
%-------------------
We start with the following result which show the existence of a
unique harmonic extension $v$:
\begin{proposition}[Existence and uniqueness for \eqref{eq:exp}]
  For all $u\in H^{\frac12}_N(\Omega)$, there exists a unique extension $v\in
  H^{1}(C)$ solution of (\ref{eq:exp}).  
  \item Furthermore, if
  $u(x)=\sum_{k=1}^{\infty} c_k \vphi_k(x)$, then
\begin{equation}\label{eq:v}
v(x,y)=\sum_{k=1}^{\infty} c_k \vphi_k(x) \exp(-\lambda_k^{\frac12}y).
\end{equation}
\end{proposition}
%--------------------
\begin{proof}
We recall that $H^{\frac12}_N(\Omega)=H^{\frac12}(\Omega)$,
    and for a given $u\in H^{\frac12}(\Omega)$ we consider the following
    minimization problem:
$$ 
\inf \left\{ \int_C |\na w|^2\, dx\, dy \,;\, w\in H^1(C)\,,\;
  w(\cdot,0)=u\mbox{ on }\Omega\right\}.
$$
Using classical arguments, one can show that this problem has a unique
minimizer $v$ (note that the set of functions on which we minimize the
functional is not empty).  This minimizer is a weak solution of
\eqref{eq:exp}. In particular, it satisfies
$$ 
\int_C \na v\cdot \na w\, dx\, dy=0
$$
for all $w\in H^1(\Omega)$ such that $w(\cdot,0)=0$ on $\Omega$, which
includes a weak formulation of the Neumann condition.

Finally, the representation formula (\ref{eq:v}) follows from a
straightforward computation. Indeed, we have
\begin{eqnarray*}
  \int_0^\infty \int_\Omega |\na v|^2\, dx\, dy & =& \int_0^\infty \int_\Omega  |\pa_x v|^2 + |\pa_y v|^2 \, dx\, dy \\
  & =& 2\sum_{k=1}^\infty c_k^2\lambda_k \int_0^\infty \exp(-2\lambda_k^{1/2}y)\, dy\\
  & =& 2 \sum_{k=1}^\infty c_k^2\lambda_k \frac{1}{2\lambda_k^{1/2}}\\
  & = & \sum_{k=1}^\infty c_k^2\lambda_k^{1/2} = ||u||^2_{\dot{H}^{\frac12}(\Omega)}
\end{eqnarray*}
which shows that $v$ belongs to $H^1(C)$. 
The fact that $v$ satisfies \eqref{eq:exp} is  easy to check.
  \end{proof}
%--------------------
  We can now show:
\begin{proposition}[The operator $I$ is of Dirichlet-to-Neumann
type] \label{prop:dir-to-neum} For all $u\in H^2_N(\Omega)$, we
  have
$$
I(u)(x)= - \frac{\pa v}{\pa\nu}(x,0) =\pa_y v(x,0)\quad \mbox{ for all $x \in \Omega$,}
$$
where $v$ is the unique harmonic extension solution of (\ref{eq:exp}).
\item Furthermore $I\circ I (u)= -\Delta u$.
\end{proposition}
%--------------------
\begin{proof}
This follows from a direct computation using (\ref{eq:v}).
Furthermore, if $u$ is in $H^2_N(\Omega)$, we know that
 $ \sum_{k=0}^{\infty} c_k^2 \lambda_k^2<\infty$. It is now
easy to derive the following equality
$$
I(I(u)) =  \sum_{k=0}^{\infty} c_k \lambda_k \vphi_k(x) = -\Delta u.
$$
\end{proof}

\paragraph{Integral representation.}
The operator $I$ can also be represented as a singular integral operator. 
Indeed, we will prove the following
%--------------------------
\begin{proposition}\label{prop:integral-rep}
Consider a smooth function $u: \Omega \to \R$. Then for all $x \in \Omega$,
$$
I(u)(x) = \int_\Omega (u(y) - u(x)) \nu (x,y) dy 
$$
where $\nu(x,y)$ is defined as follows: for all
$x,y \in \Omega$, 
\begin{equation}\label{defi:nu}
  \nu (x,y) =  \frac{\pi}2 \left( \frac{1}{1 - \cos (\pi (x-y))}+ \frac{1}{1 - \cos (\pi (x+y))} \right).
\end{equation}
\end{proposition}
%-------------------------
\begin{proof}
  We use the Dirichlet-to-Neumann definition of $I$. Let $v$ denote
  the solution of \eqref{eq:exp}. Then $v$ is the restriction to
  $(0,1)$ of the unique solution $w$ of \eqref{eq:exp} where $\Omega$
  is replaced with $(-1,1)$ and $u$ is replaced by its even extension
  to $(-1,1)$.  In particular, $w$ is even with respect to $x$. Then
  there exists a holomorphic function $W$ defined in the cylinder
  $(-1,1) \times (0,+\infty)$ such that $w = \mathrm{Re} (W)$.  Next,
  we consider the holomorphic function $z \mapsto e^{i\pi z}=e^{-\pi
    y} e^{i\pi x}$ defined on the cylinder $(-1,1) \times (0,+\infty)$
  with values into the unit disk $D_1 = \{ (x,y) : x^2 + y^2 < 1
  \}$. If $z$ denotes the complex variable $x+iy$, then a new
  holomorphic function $W_0$ is obtained by the following formula
$$
W(z) = W_0 (e^{i \pi z}).
$$
In particular, $W_0$ is defined and harmonic in $D_1$. This implies
that the function $W_0$ can be represented by the Poisson integral.
Precisely,
$$
W_0 (Z) = \frac{1 - |Z|^2}{2 \pi} \int_{\partial D_1} \frac{W_0(Y)}{|Y-Z|^2} d\sigma (Y). 
$$
This implies that for all $z \in C$,
$$
W(z) = \frac{1 - e^{-2 \pi y}}{2 \pi} \int_{-1}^1
\frac{W(\theta)}{|e^{i \pi \theta} -e^{-\pi y} e^{i\pi x}|^2} \pi d
\theta
$$
and we finally obtain
$$
w(x,y) =  \frac{1 - e^{-2 \pi y}}{2} \int_{-1}^1
\frac{w(\theta,0)}{|e^{i \pi \theta} -e^{-\pi y} e^{i\pi x}|^2}  d
\theta .
$$
Taking $w=1$, we get in particular the following equality:
$$
1 =  \frac{1 - e^{-2 \pi y}}{2 } \int_{-1}^1
\frac{1}{|e^{i \pi \theta} -e^{-\pi y} e^{i\pi x}|^2}  d
\theta .
$$
We deduce:
$$
\frac{w(x,y) - w(x,0)}y =  \frac{1 - e^{-2 \pi y}}{2 y} \int_{-1}^1
\frac{w(\theta,0) -w (x,0)}{|e^{i \pi \theta} -e^{-\pi y} e^{i\pi x}|^2}  d
\theta 
$$
which implies (letting $y$ go to zero):
$$
\partial_y w (x,0) = \pi \int_{-1}^1
\frac{w(\theta,0) -w (x,0)}{|e^{i \pi \theta} - e^{i\pi x}|^2}  d
\theta .
$$
The integral on the right hand side of the previous equality is
understood in the sense of the principal value of the associated
distribution.  We finally use the fact that $w$ is even in $x$ and is
equal to $u$ on $\Omega$ to obtain the following singular integral representation of $I(u)$:
$$
I (u)(x) = \pi \int_0^1 (u(\theta,0) -u (x,0)) \left( \frac{1}{|1 -
    e^{i\pi (x-\theta)}|^2} + \frac{1}{|1 - e^{i\pi (x+\theta)}|^2}
\right) d \theta .
$$
\end{proof}

\paragraph{The space $H^{-\frac12}(\Omega)$.} 
The space $H^{-\frac12}(\Omega)$ is defined as the topological dual
space of $H^{\frac12}(\Omega)$. It is classical that for any $u \in
H^{-\frac12}(\Omega)$, there exists $u_1 \in L^2 (\Omega)$ and $u_2
\in H^{\frac12}(\Omega)$ such that $u =u_1+ \partial_x u_2$ (in the
sense of distributions). We will also use repeatedly the following
elementary lemma, whose proof is left to the reader:
%-----------------------------------------------------------------
\begin{lemma}
  If $u \in H^{\frac12}(\Omega)$, then the distribution $I(u)$ is
in $H^{-\frac12}(\Omega)$ and for all $v \in
  H^{\frac12}(\Omega)$,
$$
\langle I (u), v \rangle_{H^{-\frac12}(\Omega),H^{\frac12}(\Omega)} 
= - \sum_{k=0}^{+\infty} \lambda_k^{\frac12} c_k d_k
$$
where $u = \sum_{k=0}^{+\infty} c_k \vphi_k$ and $v =
\sum_{k=0}^{+\infty} d_k \vphi_k$. In particular,
$$
-\langle I(u),u \rangle_{H^{-\frac12}(\Omega),H^{\frac12}(\Omega)} 
= ||u||^2_{\dot{H}^{\frac12}(\Omega)}.
$$
\end{lemma}
%-----------------------------------------------------------------

\paragraph{Important equalities.} The semi-norms
$||\cdot||_{\dot{H}^{\frac12}(\Omega)}$, $||\cdot||_{\dot{H}^{1}(\Omega)}$, $||\cdot
||_{\dot{H}^{\frac32}(\Omega)}$ and $||\cdot
||_{\dot{H}^{2}_N(\Omega)}$ are related to the operator $I$ by
equalities which will be used repeatedly.
%---------------------------------------------------------------
\begin{proposition}[The operator $I$ and several semi-norms]\label{prop:semi-norms}
\item For all $u\in H^{\frac12}(\Omega)$, we have 
$$
\frac12 \int_\Omega \int_\Omega (u(x)-u(y))^2 \nu (x,y) dx dy
 = ||u||^2_{\dot{H}^{\frac12}(\Omega)}.
$$
\item For all $u \in H^1 (\Omega)$, we have
$$
\int_\Omega (I(u))^2 dx = \|u\|^2_{\dot{H}^1 (\Omega)}.
$$
\item For all $u\in H^2_N(\Omega)$, we have
$$ 
- \int_\Omega I(u)_x u_x\, dx = ||u||_{\dot{H}^{\frac32}_N(\Omega)}^2. 
$$
\item For all $u\in H^2_N(\Omega)$, we have
$$
 \int_\Omega  (\pa_xI (u))^2 \, dx  = \|u\|_{\dot{H}^2_N (\Omega)}^2.
$$
\end{proposition}
%----------------------------------------------------------------
\begin{rem}
Note that $I(u)_x\neq I(u_x)$.
\end{rem}
%-----------
\begin{proof}
The two first equalities are easily derived form the definition of $I$,
definitions of the semi-norms, the integral representation of $I$ and 
the fact that $\nu(x,y)=\nu (y,x)$. 

In order to prove the third and fourth equalities, we first
remark that $\pa_x\vphi_k=-\lambda_k^{\frac12} \sin(k \pi x)$ form an
orthogonal basis of $L^2(\Omega)$.

In order to prove the fourth equality, we first write 
$$
\partial_x (I(u)) = -\sum_{k=1}^\infty c_k \lambda_k^{\frac12}
\partial_x \vphi_k \quad \text{ in } L^2 (\Omega)
$$ 
from which we deduce
\begin{eqnarray*} 
  \int_\Omega  (I (u)_x)^2 \, dx &  = & \sum_{k=1}^\infty c_k^2 \lambda_k \int_\Omega (\pa_x \vphi_k)^2\, dx \\
  &=& \sum_{k=1}^\infty c_k^2 \lambda_k \int_\Omega  \vphi_k (- \pa_{xx}\vphi_k)\, dx \\
  &=& \sum_{k=0}^\infty c_k^2 \lambda_k^2 =||u||_{\dot{H}^2_N}^2.
\end{eqnarray*}

As far as the third equality is concerned, we note that
$$ 
u_x = \sum_{k=0}^{\infty} c_k  \pa_x \vphi_k \quad \text{ in } L^2 (\Omega).
$$
We then have
\begin{eqnarray*}
-\int_\Omega I(u)_x u_x\, dx & = &  
\sum_{k=0}^{\infty} c_k^2 \lambda_k ^{\frac12}\int_\Omega (\pa_x \vphi_k)^2\, dx \\
& = & 
-\sum_{k=0}^{\infty} c_k^2 \lambda_k^{\frac12} \int_\Omega \vphi_k \pa_{xx} \vphi_k\, dx \\
& = & 
\sum_{k=0}^{\infty} c_k^2 \lambda_k^{\frac12} \int_\Omega \lambda_k \vphi_k ^2\, dx \\
& = & 
\sum_{k=0}^{\infty} c_k^2  \lambda_k^{\frac32}  = ||u||_{\dot{H}^{\frac32}(\Omega)}^2.
\end{eqnarray*}
\end{proof}

\subsection{The problem $-I(u)=g$}
We conclude this preliminary section by giving a few results about the following problem:
\begin{equation}\label{eq:Iinv0}
\begin{array}{cc}
\mbox{For a given $g\in L^2(\Omega)$, find $u\in H^1(\Omega)$ such that }\\[3pt]
-I(u) = g .
\end{array}
\end{equation}
Note that $\int_\Omega I(u)\, dx = 0 $ for all $u\in H^1(\Omega)$ (since $\int_\Omega \vphi_k\, dx=0$ for all $k\geq 1$) and
so a necessary condition for the existence of a solution to
(\ref{eq:Iinv0}) is
$$
\int_\Omega g(x)\, dx = 0.
$$
Note also that there is no uniqueness since if $u$ is a solution then
$u+C$ is also a solution for any constant $C$.  
We may however expect a unique solution if we add the further constraint $\int u\, dx=0$.
Indeed, a weak solution $u\in
H^{\frac{1}{2}}(\Omega)$ for $g\in H^{-\frac 1 2 }(\Omega)$ can be
found using Lax-Milgram theorem in $\{ u\in \ha\, ;\, \int_\Omega u\,
dx=0\} $ equipped with the norm $||u||_{\had}$.  Alternatively, we can
use the spectral framework: For $g\in L^2(\Omega)$ such that $\int_\Omega g(x)\, dx=0$, we have
$$ 
g=\sum_{k=1}^\infty d_k \vphi_k \quad \mbox{ with } \sum_{k=1}^\infty d_k^2 <\infty.
$$
We can then write:
\begin{equation}\label{eq:uuu}
u=I^{-1}(g):=\sum_{k=1}^\infty \frac{d_k}{\lambda_k^{\frac12}} \vphi_k
\end{equation}
which clearly lies in $H^1(\Omega)$ and satisfies $\int_\Omega u \,
dx= 0$.  The fact that the $\vphi_k$'s form an orthogonal basis of
$L^2 (\Omega)$ implies that there is only one solution of
\eqref{eq:Iinv0} such that $\int_\Omega u \,dx =0$.  Finally it is
clear from (\ref{eq:uuu}) that further regularity on $g$ will imply
further regularity on $u$.  We sum up this discussion in the following
statement
%--------------
\begin{theorem}\label{thm:basic-existence}
  For all $g \in L^2 (\Omega)$ such that $\int_\Omega g \, dx =0$, there
  exists a unique function $u \in H^1 (\Omega)$ such that $-I(u)=g$ in
  $L^2 (\Omega)$ and $\int_\Omega u \,dx = 0$.  Furthermore, if $g$ is
  in $H^1(\Omega)$, then $u\in H^2_N(\Omega)$.
\end{theorem}
%--------------------
We will also use the following corollary of the previous theorem
%--------------------------------------------------------------------------------
\begin{corollary}\label{cor:aux}
For all $g\in L^2(\Omega)$, there exists a unique solution  $u\in H^1(\Omega)$ of
$$ -I(v) + \int_\Omega v\, dx = g.$$
\end{corollary}
%--------------------------------------------------------------------------------
\begin{proof}
  Set $m=\int_\Omega g(x)\, dx$ and consider $\tilde{g} = g-m$. Then
  $\tilde{g} \in L^2 (\Omega)$ and $\int_\Omega \tilde{g} dx =
  0$. There is a (unique) $u\in H^1(\Omega)$ such that
$$
-I(u) = g-m , \quad \int_\Omega u(x)\, dx=0.
$$
We then set $v=u+m$.  Then $\int_\Omega v\, dx = m$ and
$$
-I(v)=-I(u)=g-m=g-\int_\Omega v\, dx.
$$
As far as uniqueness is concerned, if we consider two solutions $v_1$ and $v_2$ 
then we have
$$
\int_\Omega v_1 dx = \int_\Omega v_2 dx = \int_\Omega g 
$$
and this implies that $w =v_1 -v_2$ satisfies $-I(w)=0$. The
uniqueness of the solution given by Theorem~\ref{thm:basic-existence}
implies that $w=0$ and the proof is complete.
\end{proof}

\section{A regularized problem}\label{sec:regularized}
We now turn to the proof of Theorem~\ref{thm:main}.  The degeneracy of
the diffusion coefficient is a major obstacle to the development of a
variational argument.  As in \cite{bf90}, the existence of solution
for \eqref{eq:main} is thus obtained via a regularization approach:
Given $\eps>0$, we consider
\begin{equation}\label{eq:regularized}
  \partial_t u +  \partial_x ( f_\eps(u) \partial_x I (u) )  = 0 ,  
  \quad t \in (0,T), x \in \Omega 
\end{equation}
where
$$f_{\eps}(s)={s_+}^n+\eps$$
(with $s_+=\max(s,0)$), with the initial condition
\begin{equation}\label{ic:regularized}
u(0,x)=u_0(x)
\end{equation}
and boundary conditions
$$ u_x=0\, , \quad f_\eps(u)\pa_x(I(u))=0 \quad \mbox{ on } \pa\Omega.$$

The first step in the proof of Theorem~\ref{thm:main}, is to prove the
following proposition:
%-----------------------------------------------------------------------------------
\begin{proposition}[Existence of solution for the regularized problem]\label{prop:eps}
For all $u_0\in H^{\frac12}(\Omega)$ and for all $T>0$, there exists a unique function $u^\eps$ such that
$$
u^\eps\in L^{\infty}(0,T; \ha )\cap L^2(0,T;H^2_N(\Omega))
$$ 
solution of 
\begin{equation}\label{eq:eps}
   \iint_Q u^\eps \partial_t \vphi \, dx \, dt + \iint_Q f_\eps(u^\eps)
   \pa_xI (u^\eps)\pa_x \vphi \, dx\, dt
=- \int_\Omega u_0 \vphi(0,\cdot)\, dx
\end{equation}
 for all $\vphi\in \mathcal{C}^1_c ([0,T),H^1(\Omega))$ with $Q =\Omega \times (0,T)$.

Moreover,  the function $u^\eps$ satisfies
\begin{equation}\label{estim:conservation mass eps}
  \intom u^\eps (x,t)\, dx = \intom u_0 (x)\, dx\quad \text{ a.e. } t \in (0,T)
\end{equation}
and 
\begin{equation} \label{estim:nrj eps} \|u^\eps (t,\cdot)\|_{\had}^2 +
  2 \int_0^t\int_\Omega f_\eps (u^\eps) (\partial_x I (u^\eps))^2 \, dx\, ds\leq
  \| u_0 \|_{\had}^2 \; \text{ a.e. } t \in (0,T).
\end{equation}

Finally, if  $G_\eps$ is a non-negative function such
that $G_\eps''(s)=\frac{1}{f_\eps(s)}$, then $u^\eps$ satisfies
for almost every $t \in (0,T)$
\begin{equation}\label{estim:entropy eps}
  \int_\Omega G_\eps(u^\eps)(x,t)\, dx 
  + \int_0^t \|\ u^\eps(s) \|_{\dot{H}_N^{\frac32}(\Omega)}^2\, ds
  \leq \int_\Omega  G_\eps(u_0)\, dx.
\end{equation}
\end{proposition}
%--------------------------------------------------------------------

\begin{rem}
  Note that this result does not require condition
  (\ref{eq:initpositive}) to be satisfied and is thus valid with
  compactly supported initial data.  However, we will need condition
  (\ref{eq:initpositive}) to get enough compactness on $u^\eps$ to
  pass to the limit $\eps\to 0$ and complete the proof of Theorem
  \ref{thm:main}.
\end{rem}

There are several possible approaches to prove Proposition
\ref{prop:eps}.  In the next sections, we present a proof based on a
time discretization of (\ref{eq:eps}) and fairly classical
monotonicity method (though the operator here is not monotone, but
only pseudo-monotone).

\subsection{Stationary problem}

In order to prove Proposition \ref{prop:eps}, we first consider the following stationary
problem (for $\tau>0$):
\begin{equation}\label{eq:stat}  
\begin{array}{l}
\mbox{For a given $g \in H^{\frac12}(\Omega)$, find $u \in H^2_N(\Omega)$ such that}\\[5pt]
\left\{\begin{array}{lll}
 u+  \tau \partial_x (  f_{\eps}(u) \partial_x I( u) ) &=& g\quad  \mbox{ in }\Omega\\[5pt]
\pa_x u=0 \text{ and } \pa_xI(u)&=&0\quad  \text{ on } \pa\Omega.
 \end{array}\right. 
\end{array}
\end{equation}
Once we prove the existence of a solution for (\ref{eq:stat}), a
simple time discretization method will provide the existence of a
solution to (\ref{eq:eps}).  We are going to prove:
%------------------------------------------------------------------
\begin{proposition}[The stationary problem] \label{prop:stationary}
  For all $g \in \ha$, there exists $u \in H^2_N(\Omega)$ such that
  for all $\vphi \in H^1(\Omega)$,
\begin{equation}\label{eq:stationary}
  \frac1\tau \intom (u -g) \vphi\, dx - \intom f_\eps (u) \, \partial_x I(u) \,\pa_x\vphi \, dx= 0 \, .
\end{equation}
Furthermore,
\begin{eqnarray}
\label{estim:conservation mass stationary}
\intom u (x)\, dx &=& \intom g(x)\, dx \, , \\
\label{estim:stationary}
  \|u\|^2_{\had} + 2\tau \int_\Omega f_\eps (u) (\partial_x Iu)^2 \, dx &\le&
  \|g\|^2_{\had} \, ,
\end{eqnarray}
and if $   \int_\Omega G_\eps(g) \, dx<\infty$ then
\begin{eqnarray}
\label{eq:Faux} 
  \intom G_\eps(u) \, dx+ \tau \|u\|^2_{\dot{H}_N^{\frac32}(\Omega)} 
  &\le& \int_\Omega G_\eps(g)\, dx.
\end{eqnarray}
\end{proposition}
%------------------------------------------------------------------
In order to prove such a result, we have to reformulate (\ref{eq:stationary}): 

\paragraph{New formulation of \eqref{eq:stationary}.}
We are going to use classical variational methods to show the
existence of a solution to (\ref{eq:stationary}).  In order to work with a
coercive non-linear operator, we need to take $\vphi=-I(v)$ as a test function. We
note, however, that by doing that we would restrict ourself to test
functions with zero mean value.  In order to recover all test
functions from $H^1(\Omega)$, we use Corollary~\ref{cor:aux} and
consider
\begin{equation}\label{eq:vphi}
\vphi = -I(v) +\int_\Omega v\, dx
\end{equation}
for some function $v\in H^2_N(\Omega)$. Let us emphasize the fact that
Corollary~\ref{cor:aux} implies in particular that there is a
one-to-one correspondence between $\vphi\in H^1(\Omega)$ and $v\in
H^2_N(\Omega)$ satisfying (\ref{eq:vphi}).

Using (\ref{eq:vphi}), Equation (\ref{eq:stationary}) becomes:
\begin{eqnarray}
  &&\!\!\!\!\!\!\!\!- \int_\Omega u \, I(v) \, dx  
+ \left(\int_\Omega u\, dx\right)\left( \int_\Omega v\, dx\right)+ \tau \int_\Omega f_\eps(u)
  \pa_xI (u)\pa_x I(v) \, dx \nonumber \\
  && \qquad \qquad\qquad\qquad \qquad = - \int_\Omega g\, I(v) \, dx 
+ \left(\int_\Omega g\, dx\right)\left( \int_\Omega v\, dx \right).\label{eq:vstat}
\end{eqnarray}
We can now introduce the non-linear operator we are going to work with.

\paragraph{A non-linear operator.}
We define for all $u$ and $v\in H^2_N(\Omega)$ 
$$
A(u)(v)=- \int_\Omega u \, I(v) \, dx  + \left(\int_\Omega u\, dx\right)\left( \int_\Omega v\, dx\right)
+ \tau \int_\Omega f_\eps(u) \pa_xI (u)\pa_x I(v) \, dx.
$$
One can now show that this non-linear operator is continuous, coercive and
pseudo-monotone. Classical theorems imply the existence of a solution to the equation $A(u) =g$ for proper $g$'s. More precisely, we
have the following proposition:
%---------------------------------------------------------- 
\begin{proposition}[Existence for the new problem]\label{prop:stat}
For all $g\in \ha$ there exists $u\in H^2_N(\Omega)$ such that
\begin{equation}\label{eq:Au}
  A(u)(v)= - \int_\Omega g\, I(v) \, dx + \left(\int_\Omega g\, dx\right)\left( \int_\Omega v\, dx \right)
\quad \mbox{ for all } v\in  H^2_N(\Omega).
\end{equation}
\end{proposition}

For the sake of readability, we postpone the proof of this rather technical proposition to Appendix \ref{app:1}, and we turn to the proof of Proposition~\ref{prop:stationary}.
%-------------------------------------------------------
\begin{proof}[Proof of Proposition~\ref{prop:stationary}]
  For a given $g\in \ha$, Proposition \ref{prop:stat} gives the
  existence of a solution $u\in H^2_N(\Omega)$ of (\ref{eq:vstat}).  We recall that
  for any $\vphi\in H^1(\Omega)$, there exists $v\in H^2_N(\Omega)$ such that
$$
\vphi = -I(v)+\int_\Omega v\, dx
$$
and so equivalently, we have that $u$ satisfies (\ref{eq:stationary}) for
all $\vphi\in H^1(\Omega)$. 

Next, we note that the mass conservation equality
\eqref{estim:conservation mass stationary} is readily obtained by
taking $v=1$ as a test function in (\ref{eq:vstat}), while
\eqref{estim:stationary} follows by taking $v=u-\int_\Omega u\, dx$:
\begin{multline*}
  \|u\|^2_{\had} + \tau \int_\Omega f_\eps (u) |\pa_xI(u)|^2    =   -\int_\Omega g I (u)\, dx   \\ 
  \le  \| g \|_{\had} \|u \|_{\had} 
   \leq  \frac{1}{2} ||g||^2_{\had}+ \frac{1}{2} ||u||^2_{\had} \, .
\end{multline*}

Finally since $G'_\eps$ is smooth with $G'_\eps$ and $G''_\eps$
bounded and $\Omega$ is bounded, we have $G'_\eps(u) \in H^{1}
(\Omega)$.  We can thus find $v\in H^2_N(\Omega)$ such that
$$ -I(v)+\int_\Omega v(x)\, dx = G'_\eps(u).$$
Equation (\ref{eq:vstat}) then implies:
$$
-\int_\Omega u G_\eps'(u) \, dx+ \tau \int_\Omega f_\eps(u) \, F''_\eps
(u) \, \pa_x I(u) \, \pa_x  u \, dx
=- \int _\Omega g G_\eps'(u)\, dx
$$
and so (using the definition of $G_\eps$ given in Proposition~\ref{prop:eps})
$$
- \tau \int_\Omega \pa_x I(u) \; \pa_x u   \, dx = \int _\Omega
G_\eps'(u)(g-u)\, dx
$$
Since $G_\eps$ is convex ($G_\eps''\geq 0$), we have 
$G_\eps'(u)(g-u) \le G_\eps(g)-G_\eps(u)$ and we deduce
\eqref{eq:Faux} (using Proposition~\ref{prop:semi-norms}).
\end{proof}

\subsection{Proof of Proposition \ref{prop:eps} }

In order to construct the solution $u^\eps$ of~\eqref{eq:regularized}, 
we discretize the problem with respect to $t$, and  construct
a piecewise constant function
$$ 
u^\tau (x,t) = u^n (x) \text{ for } t \in (n\tau, (n+1)\tau), n \in \{0, \dots, N+1\},
$$ 
where $\tau= T/N$ and $(u^n)_{n\in\{0,\dots, N+1\}}$ is such that
$$
\frac1\tau (u^{n+1}-u^n) + \partial_x ( f_\eps(u^{n+1}) \partial_x
I (u^{n+1}) ) = 0 \, .
$$

The existence of the $u^n$ follows from Proposition \ref{prop:stationary} by induction on $n$.
We deduce:
%----------------------------------------------------------------------------------------------------
\begin{corollary}[Discrete in time approximate solution]\label{cor:approx sol}
  For any $N>0$ and $u_0^\eps \in \ha$, there exists a function
  $u^\tau \in L^\infty (0,T; H^{\frac12}(\Omega))$ such that
\begin{itemize}
\item $t \mapsto u^\tau (x,t)$ is constant on $[k\tau,(k+1)\tau)$ for
  $k \in \{0, \dots, N\}$, $\tau = \frac{T}N$,
\item
$u^\tau  = u_0$ on $[0,\tau) \times \Omega$,
\item
for all $\varphi \in \mathcal{C}^1 (0,T,H^1(\Omega))$,
\begin{equation}\label{eq:approx discrete}
\iint_{Q_{\tau,T}} \frac{u^\tau  - S_\tau u^\tau}{\tau} \varphi  \, dx \, dt=
 \iint_{Q_{\tau,T}} f_\eps (u^\tau )\partial_x I (u^\tau) \partial_x 
\varphi \, dx\, dt
\end{equation}
where $Q_{\tau,T} = (\tau,T) \times \Omega$ and $S_\tau u^\tau (x,t) = u^\tau (t-\tau,x)$. 
\end{itemize}
Moreover, the function $u^\tau$ satisfies
\begin{equation}
\label{estim:conservation mass tau}
\intom u^\tau (x,t) \, dx= \intom u_0(x)\,dx \quad \text{ a.e. } t \in (0,T)
\end{equation}
and for all $t \in (0,T)$
\begin{equation} \label{estim:nrj tau}
\|u^\tau (t,\cdot)\|_{\had}^2 + 2  \int_{0}^t\int_\Omega f_\eps (u^\tau) (\partial_x I (u^\tau))^2 \, dx\, dt \le \| u_0 \|_{\had}^2 
\end{equation}
and if $ \int_\Omega G_\eps(u_0) \, dx <\infty$, then for all $t \in (0,T)$
\begin{equation}\label{estim:entropy tau}
  \int_\Omega G_\eps(u^\tau (t,\cdot)) \, dx +
  \int_{0}^{t}\|u^\tau\|^2_{\dot{H}^{\frac32}_N(\Omega)}\, ds
   \le \int_\Omega G_\eps(u_0) \, dx .
\end{equation}
\end{corollary}
%----------------------------------------------------------------------------------------------------
%-----------------------------------

It remains to prove that $u^\tau$ converges to a solution of \eqref{eq:eps} as $\tau$ goes to zero.
This is fairly classical and we detail the proof for the interested reader in Appendix \ref{app:B}.

\section{Proof of Theorem~\ref{thm:main}: Case $n \ge 2$}\label{sec:main-proof}

  Proposition~\ref{prop:eps} provides the existence of a solution
  $u^\eps \in L^\infty(0,T;\ha)\cap L^2(0,T;H^2_N(\Omega))$ of
  \eqref{eq:eps}.  Our goal is now to pass to the limit $\eps\to0$. We
  point out that at this point, the solution $u^\eps$ may change sign
  and that it is only at the limit $\eps\rightarrow 0$ that we are
  able to recover a non-negative solution, using the fact that $n\geq 2$.

\paragraph{Step 1: Compactness result.}
First, we note that (\ref{estim:nrj eps}) implies 
\begin{equation}\label{eq:ueH}
||u^\eps(t)||_{H^{\frac12}(\Omega)}\leq ||u_0(t)||_{H^{\frac12}(\Omega)} \quad \mbox{ for all } \eps>0.
\end{equation}

The bound \eqref{eq:ueH} and Sobolev embedding theorems imply
that the sequence $(u^\eps)_{\eps>0}$ is bounded in $L^\infty(0,T;L^{p}(\Omega))$ for
all $p<\infty$ and so $ f_\eps(u^\eps)$ is bounded in
$L^\infty(0,T;L^{p}(\Omega))$ for all $p<\infty$. 
Furthermore, (\ref{estim:nrj eps}) also gives that $f_\eps(u^\eps)^{\frac 1 2} \pa_x I(u^\eps)$ is bounded in $L^2(0,T;L^2(\Omega))$.
We deduce that
$$ f_\eps(u^\eps) \pa_x I(u^\eps) \quad \mbox{ is bounded in } L^2(0,T;L^r(\Omega))$$
for all $r\in [1,2)$.
Writing
$$ \pa_t u^\eps = -\pa_x(f_\eps(u^\eps) \pa_x I(u^\eps)) \qquad \mbox{ in } \mathcal D'(\Omega),$$
we deduce that  $(\partial_t u^\eps)_{\eps>0}$ is bounded in $L^2(0,T;W^{-1,r'}(\Omega))$ for all
$r' \in (2,+\infty)$.

Thanks to the following  embeddings
$$
\ha \hookrightarrow L^q(\Omega) \rightarrow W^{-1,r'}(\Omega)
$$
for all $q<\infty$, if follows (using Aubin's lemma) that 
 $(u^\eps)_{\eps>0}$ is relatively compact in
$\mathcal{C}^0 (0,T,L^q(\Omega))$ for all $q < +\infty$.
Hence, we can extract a subsequence, still denoted by $u^\eps$ such that 
$$u^\eps \lto u \quad \mbox{ in 
$ \mathcal{C}^0 (0,T,L^q(\Omega))$ for all $q<\infty$} $$
and 
$$u^\eps \lto u \quad \mbox{ almost
everywhere in $Q$.}$$

\paragraph{Step 2: Passing to the limit in Equation~\eqref{eq:eps}.}
We now have to pass to the limit in  \eqref{eq:eps}.
We fix $\vphi \in \mathcal{D} (Q)$.
Since $u^\eps \to u$ in $\mathcal{C}^0
(0,T,L^1 (\Omega))$, we have
$$
\iint_Q u^\eps \partial_t \vphi \, dx \, dt \to \iint_Q u
\; \partial_t \vphi\, dx \, dt.
$$

Next, we remark that%, since $\ha\subset L^2(\Omega)$,
\eqref{estim:nrj eps} implies
$$
\eps \iint_Q (\partial_x I (u^\eps))^2  \, dx \, dt%+ \eps \iint_Q (\partial_x u^\eps)^2 
\leq\frac{1}{2} ||u_0||_{H^{\frac 1 2 }(\Omega)}.
$$  
Cauchy-Schwarz inequality thus implies
\begin{eqnarray*}
\iint_Q \eps \partial_x I (u^\eps) \partial_x  \vphi  \, dx \, dt& \lto & 0 \, .
\end{eqnarray*}

Finally, \eqref{estim:nrj eps} implies that $(u^\eps)_+^{\frac n 2} \partial_x I (u^\eps)$ is bounded
in $L^2(0,T,L^2(\Omega))$. Since $u^\eps$ is bounded in
$L^\infty(0,T;L^p(\Omega))$ for all $p<\infty$ we deduce that $
(u^\eps)_+^n \partial_x I (u^\eps) $ is bounded in $L^2(0,T;
L^r(\Omega))$ for all $r\in[1,2)$ and so
$$ 
h^\eps:= (u^\eps)_+^n \partial_x I (u^\eps) \rightharpoonup h \quad
\mbox{ in $L^2(0,T; L^r(\Omega))$-weak.}
$$
Passing to the limit in \eqref{eq:eps}, we get:
$$   \iint_Q u \; \partial_t \vphi\, dx\, dt + \iint_Q h \; \pa_x \vphi\, dx\,dt  = 
- \iint_Q u_0 \vphi(0,\cdot)\, dx\,dt
$$
for all $\vphi \in \mathcal{D} (\overline{Q})$.  

\paragraph{Step 3: Equation for the flux $h$.}
In order to get
\eqref{eq:weak}, it only remains to show that
$$
h=u_+^n \partial_x I (u)
$$
in the following sense:
\begin{equation}\label{eq:hh}
 \iint_Q h \; \phi\, dx\,dt = -\iint_Q  n u_+^{n-1} \pa_x u\; I (u)  \; \phi\, dx\,dt 
-\iint_Q   u_+^n  I (u)  \; \pa_x \phi\, dx\,dt
\end{equation}
for all test function $\phi$ such that $\phi|_{\pa\Omega}=0$; that is
$$
h=\pa_x(u_+^n I(u))-nu_+^{n-1}( \pa_x u) I(u) \quad 
\text{ in } \mathcal D'(\Omega).
$$
For that we note that since
$$ 
\int_\Omega G_\eps(u_0)\, dx\leq C,
$$
Inequality~\eqref{estim:entropy eps} implies that $(u^\eps)_\eps$ is
bounded in $L^2(0,T;H^{\frac{3}{2}}(\Omega))$. Recall that
$(\partial_t u^\eps)_\eps$ is bounded in $L^2(0,T, W^{-1,r'}
(\Omega))$ for all $r' \in (2,+\infty)$.  Aubin's lemma then implies that
$u^\eps$ is relatively compact in $L^2(0,T;H^s(\Omega))$ for $s<3/2$.
In particular, we can assume that
$$ I(u^\eps)\lto I(u) \quad \mbox{ in } L^2(0,T;L^2(\Omega))$$
and 
$$ \pa_x u^\eps \lto \pa_x u \quad \mbox{ in } L^2(0,T;L^p(\Omega)), \mbox{ for all } p<\infty.$$
Writing
\begin{eqnarray*}
  \iint_Q h^\eps \phi  & = &  \iint_Q  (u^\eps)_+^n \partial_x I (u^\eps)  \; \phi\, dx\,dt \\
  & = & -\iint_Q  n (u^\eps)_+^{n-1} \pa_x u^\eps\; I (u^\eps)  \; \phi\, dx\,dt 
  -\iint_Q   (u^\eps)_+^n \; I (u^\eps)  \; \pa_x \phi\, dx\,dt,
\end{eqnarray*}
we see that those estimates, together with the fact that $ u^\eps$ 
converges to $u $ in $L^\infty(0,T;L^p(\Omega))$ for all $p<\infty$,  are enough to
pass to the limit and get \eqref{eq:hh}.
\bigskip

\paragraph{Step 4: Properties of $u$.} 
It is readily seen that $u$ satisfies the conservation of mass
\eqref{estim:conservation mass} (by passing to the limit in
\eqref{estim:conservation mass eps}), and the lower semicontinuity of
the norm implies the entropy inequality \eqref{estim:entropy}.

Next, Inequality \eqref{estim:nrj eps} implies that
$g^\eps=(u_+^\eps)^{\frac n 2}\pa_xI(u^\eps)$ converges weakly in
$L^2((0,T)\times\Omega)$ to a function $g$, and the lower
semicontinuity of the norm implies \eqref{estim:nrj}.  Proceeding as
above we can easily show that
$$
g = \pa_x (u_+^{\frac n 2} I(u))-\frac n 2 u_+^{\frac n 2-1}\pa_x u\;
I(u) \quad \text{ in } \mathcal D'(\Omega).
$$

\paragraph{Step 5: non-negative solutions.}

It remains to prove that $u$ is non-negative.  This will be a
consequence of \eqref{estim:entropy eps} and the fact that $n\geq 
2$. Indeed, we recall that for all $t$ we have
\begin{equation}\label{eq:Ge}
\intom G_\eps (u^\eps(t)\, dx \le \intom G_\eps (u_0) \, dx
\end{equation}
where is such that $G_\eps''(s) = \frac1{(s_+)^n + \eps}$.  As noted
in the introduction, we can take
$$
G_\eps (s) = \int_1^s \int_1^r G''_\eps (t)\, dt\, dr
$$
which is a nonnegative convex function for all $\eps$. Noticing that
we can also write $G_\eps(s)= \int_s^1 \int_r^1 G''_\eps (t) \, dt\, dr$
when $s\leq1$, it is readily seen that $G_\eps(s)$ is decreasing with
respect to $\eps$ (so $G_\eps(s)\leq G_0(s)$ for all $\eps>0$).
Hence
$$
\intom G_\eps (u_0)\, dx \le \int_\Omega G_0 (u_0)\, dx < + \infty \,
.
$$
We deduce (using (\ref{eq:Ge})):
\begin{equation}\label{eq:limsupF}
\limsup_{\eps \to 0} \intom G_\eps (u^\eps(t)) \, dx < + \infty\, .
\end{equation}

Next, we recall that $u^\eps(\cdot,t)$ converges strongly in
$L^p(\Omega)$ to $u(\cdot,t)$. We can thus assume that it also
converges almost everywhere.  Egorov's theorem then implies the
existence of a set ${A}_\eta \subset \Omega$ such that $u^\eps
(\cdot,t) \to u(\cdot,t)$ uniformly in ${A}_\eta$ and $|\Omega
\setminus {A}_\eta|<\eta$. For some $\delta>0$, we now consider
$$
C_{\eta,\delta} = {A}_\eta \cap \{ u(\cdot,t) \le - 2 \delta \}. 
$$
For every $\eta,\;\delta>0$ there exists $\eps_0(\eta,\delta)$ such
that if $\eps \le \eps_0(\eta,\delta)$, then $u^\eps (\cdot,t) \le
-\delta$ in $C_{\eta,\delta}$.

But this implies that $C_{\eta,\delta}$ has measure zero.  Indeed, if
not, then for $\eps \le \eps_0(\eta,\delta)$ we have
$$ 
G_\eps (u^\eps(x,t)) \geq G_\eps (-\delta) \longrightarrow
G_0(-\delta) =+\infty \text{ for all } x\in C_{\eta,\delta}
$$
(we use here the assumption $n \ge 2$)
and by Fatou lemma, we get
$$ 
\liminf_{\eps\to 0} \int_{C_{\eta,\delta}} G_\eps(u^\eps(x,t))\, dx
\geq \int_{C_{\eta,\delta}} \liminf_{\eps\to 0} G_\eps(u^\eps(x,t))\,
dx = +\infty
$$
which contradicts \eqref{eq:limsupF}.

We deduce that for all $\delta>0$ and all $\eta>0$ we have
$$
| \{ u(\cdot,t) \le - 2 \delta \}| \leq |C_{\eta,\delta}|
+|\Omega\setminus A_\eta|\leq \eta
$$
and so $ | \{ u(\cdot,t) \le - 2 \delta \}| =0 $  for all $\delta>0$.
We can  conclude that 
$$
\{u(\cdot,t) < 0 \} = \bigcup_{n \ge 1} \left\{ u (\cdot,t) < - \frac1n \right\}
$$ 
has measure zero and so $u(x,t)\geq $ a.e. $x\in\Omega$ and for all $t>0$.

\section{Proof of Theorem~\ref{thm:main}: Case $n\in [1,2)$}\label{sec:other-proof}
When $n\in [1,2)$ the entropy inequality cannot be used to prove that $\lim_{\eps\to 0} u^\eps$ is non-negative.
We thus  proceed as in Bertozzi and Pugh \cite{BP2}: Introducing
$$ 
f_\delta (s) = \frac{s^{3+n}}{\delta s^n +s^{3}}
$$
and
$$u^\delta_0(x)=u_0(x)+\delta.$$
For $n<2$, we have $f_\delta(s)\sim s^3/\delta$ as $s\to 0$, and so
the corresponding entropy $G_\delta$, defined by
$$
G_\delta (s) = \int_1^s \int_1^r \frac{1}{f_\delta(t)} \, dt\, dr =
\int_1^s \int_1^r \frac \delta {t^3}+\frac 1 {t^n} \, dt\, dr.
$$
satisfies
$$
G_\delta (s) = \delta\left(\frac{1}{2s}+\frac s 2 -1\right) + G_0(s)
$$
where $G_0(s)$ is bounded in the neighborhood of $s=0$.  It is thus
readily seen that there exists $C$ such that
$$
\int_\Omega G_\eps (u^\delta_0(x))\, dx <C.
$$
Furthermore, we have $G_\delta(0)=+\infty$, so the proof developed in
the previous section (regularizing the equation by introducing
$f_{\delta,\eps}(s)= f_\delta(s)+\eps$) implies the existence of a
non-negative solution $u^\delta$ of
$$
\pa_t u^\delta + \pa_x( f_\delta(u^\delta) \pa_x I(u^\delta))=0
$$
satisfying the usual inequalities (mass conservation, energy and entropy inequality).
\medskip

Proceeding as in the previous section, we can now show that the
sequence $(u^\delta)_{\delta>0}$ is relatively compact in
$\mathcal{C}^0 (0,T,L^q(\Omega))$ for all $q<+\infty$ and in
$L^2(0,T;H^s(\Omega))$ for $s<3/2$.  In particular, we can extract a
subsequence, still denoted $u^\delta$, such that
\begin{eqnarray*}
u^\delta \lto u &&  \mbox{ in 
$ \mathcal{C}^0 (0,T,L^q(\Omega))$ for all $q <+\infty$} \\
u^\delta \lto u && \mbox{ almost
everywhere in $Q$}\\
I(u^\delta)\lto I(u) &&  \mbox{ in } L^2(0,T;L^2(\Omega))\\
 \pa_x u^\delta \lto \pa_x u && \mbox{ in } L^2(0,T;L^p(\Omega)), \mbox{ for all } p<+\infty.
 \end{eqnarray*}
Furthermore, since $u^\delta\geq 0$ for all $\delta>0$, we have 
$$ u \geq 0 \quad\mbox{ a.e. } (x,t)\in\Omega\times(0,T).$$
\medskip

In order to pass to the limit in the equation, we mainly need to check
that $f_\delta(u^\delta)$ (respectively $f'_\delta(u^\delta)$)
converges to $u^n$ (respectively $n u^{n-1}$) in $L^p(\Omega)$.  This
is a direct consequence of the convergence almost everywhere of
$u^\delta$, Lebesgue dominated convergence theorem and the fact that
$$ 
f_\delta(s) \leq s^n \mbox{ for all $s\geq 0$}
$$
and
$$
f_\delta'(s) = \frac{3\delta s^{2+2n}+n s^{5+n}}{(\eps s^n+s^3)^2}
\leq (n+2)s^{n-1} \mbox{ for all $s\geq 0$}
$$
(note that we need $n\geq 1$ to complete this computation).

This complete the proof of Theorem \ref{thm:main}.

\section{Proof of Theorem \ref{thm:2}}\label{sec:thm2-proof}
In this section, we prove Theorem \ref{thm:2}. We recall that $n>3$ and we consider the sequence $u^\eps$ of solution of the regularized equation (\ref{eq:regularized}) introduced in the proof of Theorem \ref{thm:main}.

We recall that inequality~\eqref{estim:entropy eps} implies that $(u^\eps)_\eps$ is
bounded in $L^2(0,T;H^{\frac{3}{2}}(\Omega))$. Since $(\partial_t
u^\eps)_\eps$ is bounded in $L^2(0,T, W^{-1,r'} (\Omega))$ for all $r'\in(2,+\infty)$, Aubin's lemma implies that $u^\eps$ converges strongly in
$L^2(0,T ;\mathcal C^\alpha(\Omega))$ for all $\alpha<1$.  We can thus
find a subsequence such that $u^\eps(t)$ converges strongly in
$\mathcal C^\alpha(\Omega) $ for almost every $t$ (that is for all
$t\in P$, where $|(0,T)\setminus P|=0$).

Next, we note that for $t\in P$, $u$ is actually strictly positive. Indeed, if 
$u(x_0,t_0)=0$, then for any $\alpha<1$, there is a constant $C_\alpha$ such that 
$$u(x,t) \leq C|x-x_0|^\alpha$$
We deduce
$$ \int G(u(x,t_0))\, dx \geq \int \frac{1}{(C_\alpha|x-x_0|^\alpha)^{n-2}}\, dx$$
Given $n>3$ we can choose $\alpha<1$ such that $\alpha(n-2)>1$. We
deduce
$$  \int G(u (x,t_0))\, dx =\infty$$
which contradicts (\ref{estim:entropy eps}).
\medskip

We deduce that there exists $\delta >0$ (depending on $t$) such that
for $\eps$ small enough
$$ 
u^\eps(\cdot, t) \geq \delta \mbox{ in } \Omega.
$$

Next, we note that after removing another set of measure zero to $P$,
we can always assume that
$$\liminf_{\eps \rightarrow 0}  \int f_\eps(u^\eps ) |\pa_x I(u^\eps)|^2 \, dx <\infty \quad \mbox{ for all }t\in P.$$
Indeed, if we denote
$$ A_k=\{ t\in P\, ;\,  \liminf_{\eps \rightarrow 0}  \int f_\eps(u^\eps ) |\pa_x I(u^\eps)|^2 \, dx \geq k\}$$
we have (using Fatou's lemma):
\begin{eqnarray*} 
C & \geq &  \liminf_{\eps \rightarrow 0}  \int_0^T \int f_\eps(u^\eps ) |\pa_x I(u^\eps)|^2 \, dx \, dt \\
& \geq &  \liminf_{\eps \rightarrow 0}  \int_{A_k} \int f_\eps(u^\eps ) |\pa_x I(u^\eps)|^2 \, dx \, dt \\
& \geq &  \int_{A_k}   \liminf_{\eps \rightarrow 0} \int f_\eps(u^\eps ) |\pa_x I(u^\eps)|^2 \, dx \, dt \\
& \geq &k|A_k|
\end{eqnarray*}
We deduce $|A_k|\leq C/k$ and thus
$$
| \{  t\in P\, ;\,  \liminf_{\eps \rightarrow 0}  \int f_\eps(u^\eps ) |\pa_x I(u^\eps)|^2 \, dx =\infty\}|=0.
$$
\medskip

It follows that for $t\in P$, we have
$$ \liminf_{\eps \rightarrow 0} |\pa_x I(u^\eps)|^2 \, dx <\infty$$
and so
$$ u^\eps (\cdot,t) \rightharpoonup u(\cdot,t) \quad \mbox{ in } H^2_N(\Omega)\mbox{-weak} .$$

In particular we can pass to the limit in the flux $J_\eps = f_\eps(u^\eps) \pa_x I(u^\eps)$ and write
$$ \lim_{\eps\to 0} J_\eps = J = f(u)\pa_x I(u)  \qquad \mbox{ in } L^1(\Omega),\mbox{ a.e. $t\in(0,T)$ } .$$
Furthermore, we note that we recover the boundary condition in the classical sense: 
$$ u_x (x,t)= 0 \quad \mbox{ for $x\in\pa\Omega$ and a.e. $t\in (0,T)$.}$$

\appendix

\section{Proof of Proposition \ref{prop:stat}} \label{app:1}
We denote
$$V=H^2_N(\Omega).
$$
For any $u\in V$ the functional
  $A(u)$ is clearly linear on $V$ and since $V$ is continuously
  embedded in $L^\infty(\Omega)$, we have
\begin{equation}\label{eq:cont}
  |A(u)(v)| \leq \left[||u||_{H^{\frac12}(\Omega)}
    +  \tau (\eps+||u||_{V}^3)||u||_{V} \right] ||v||_{V} \, .
\end{equation}
(Note that we used Proposition~\ref{prop:semi-norms} to get this
inequality).  The non-linear operator $A$ is thus well-defined as a
map from $V$ to $V'$. Moreover, it is bounded.

Next, we remark that we have
$$ 
A(u)(u)\geq - \int_\Omega u \, I(u) \, dx  + \left(\int_\Omega u\, dx\right)^2
+ \eps \int_\Omega  |\pa_xI (u)|^2 \, dx.
$$
We deduce from Proposition~\ref{prop:semi-norms} that
\begin{equation}\label{eq:coer}
A(u)(u)\geq \tau \eps ||u|_{H^2_N(\Omega)}^2.
\end{equation}
In particular, we have
$$
\frac{A(u)(u)}{\|u\|_V} \to +\infty \quad \mbox{  as } \|u\|_V
\to +\infty.
$$ 
The operator $A$ is thus coercive.  Proposition~\ref{prop:stat} will
now be a consequence of classical theorems if we prove that $A$ is a
pseudo-monotone operator.  Since we already know that $A$ is bounded, 
it remains to prove  the following lemma:
%-----------------------------------------------------------------------
\begin{lemma}[$A$ is pseudo-monotone]\label{lem:pseudo}
  Let $u_j$ be a sequence of functions in $V$ such that $u_j
  \rightharpoonup u$ weakly in $V$.  Then
$$
\liminf_j \; A(u_j) (u_j-v) \ge A(u) (u-v). 
$$
\end{lemma}
%-----------------------------------------------------------------------
Before we prove this lemma, let us notice that for $g \in H^{\frac12} (\Omega)$, the application
$$ 
T_g:v\mapsto -\int_\Omega g I( v) \, dx +\left(\int_\Omega g\, dx \right) \left(\int_\Omega v\, dx \right)
$$
belongs to $V'$.
Hence, using Theorem~2.7 (page~180) of \cite{lions69}, we deduce that
for all $g \in H^{\frac12}(\Omega)$, there exists a function $u \in
V$ such that $A (u) = T_g$ in $V'$, which completes the proof of Proposition \ref{prop:stat}.

\vspace{20pt}

It remains to prove Lemma~\ref{lem:pseudo}.
%---------------------------------------------
\begin{proof}[Proof of Lemma \ref{lem:pseudo}]
We first write
\begin{eqnarray}
  A(u_j) (u_j-v) &=& -\intom u_j I(u_j-v) \, dx
+ \left(\intom u_j\, dx \right) \left(\intom (u_j-v) \, dx \right)   \nonumber\\
  & & + \tau \intom f_\eps (u_j) \pa_x I (u_j)  
  \pa_x I(u_j - v) \, dx \nonumber \\ 
  &=&   \| u_j\|^2_{H^{\frac12}(\Omega)}  - \langle u_j, v \rangle_{H^{\frac{1}{2}}}\nonumber\\
  && + \tau \intom f_\eps (u_j) (\pa_x I (u_j))^2  -  \tau \intom f_\eps (u_j) \pa_x (I  u_j) \pa_x (I v)
\label{eq:pseudo-proof}
\end{eqnarray}
where
$$ 
\langle u, v \rangle_{H^{\frac{1}{2}}} = -\int_\Omega uI(v)\, dx
+\left(\int_\Omega u\, dx\right)\left( \int_\Omega v\, dx\right).
$$
We need to check that we can pass to the limit in each of those terms.

Since $u_j$ converges weakly in $V$ we immediately get
$$
\liminf_{j \to +\infty} \| u_j\|^2_{H^{\frac12}(\Omega)} \geq \| u\|^2_{H^{\frac12}(\Omega)} 
$$
and
$$
\lim_{j\to+\infty} \langle u_j, v \rangle_{H^{\frac{1}{2}}} = - \langle
u,v \rangle_{H^{\frac{1}{2}}} .
$$
Since $u_j$ is bounded in $H^2_N(\Omega)$, it is compact in $L^\infty(\Omega)$, and so
$f_\eps(u_j)$ converges to $f_\eps(u)$ strongly in $L^\infty(\Omega)$.  We thus
write
\begin{eqnarray*}
   \intom f_\eps (u_j) (\pa_x I (u_j))^2 & = & 
  \intom \big( f_\eps(u_j)-f_\eps (u) \big) (\pa_x I(u_j))^2 + \intom  f_\eps(u) (\pa_xI(u_j))^2 \\
  & \ge & - \|f_\eps(u_j)-f_\eps (u) \|_{L^\infty(\Omega)} \|u_j\|^2_V +  \intom  f_\eps (u) (\pa_x I(u_j))^2.
\end{eqnarray*}
The first term goes to zero and we have
$$
\sqrt{f_\eps(u)} \pa_xI( u_j) \rightharpoonup \sqrt{f_\eps(u)} 
\pa_x I( u) \mbox{ in } L^2 (\Omega).
$$
Again, the lower semicontinuity of the $L^2$-norm gives
 $$\lim_{j\to\infty} \tau \intom f_\eps (u) (\pa_x I (u_j))^2 \geq  \intom  f_\eps (u) (\pa_x I(u))^2.$$
Finally, we have
\begin{eqnarray*}
  f_\eps (u_j) \pa_x I (v) & \to & f_\eps (u) \partial_x I (v) \quad  \mbox{ in } L^2 (\Omega) \mbox{ strong}, \\
  \pa_x I (u_j)& \rightharpoonup & \pa_x I (u) \hspace{11mm} \mbox{ in } L^2(\Omega) \mbox{ weak}
\end{eqnarray*}
which gives the convergence of the last term in
(\ref{eq:pseudo-proof}) and completes the proof of the lemma.
\end{proof}

\section{Proof of Proposition~\ref{prop:eps}} \label{app:B}
%---------------
The proof of  Proposition~\ref{prop:eps} is divided in three steps. 

\paragraph{Step 1: \emph{a priori} estimates.}
We summarize the a priori estimates in the following lemma:
%--------------------------------------
\begin{lemma}[\textit{A priori} estimates] \label{lem:a priori tau}
The solution $u^\tau$ constructed in Corollary~\ref{cor:approx sol}
satisfies 
\begin{eqnarray}
\label{estim:approx1}
\| u^\tau \|_{L^\infty(0,T,\ha)} &\le& \| u_0^\eps
\|_{\ha} \, ,\\
 \label{estim:approx3}
\sqrt{\eps} \| \partial_x I (u^\tau) \|_{L^2(Q)} 
&\le& C \, , \\
\label{estim:approx2}
\left\| \frac{u^\tau - S_\tau u^\tau}\tau \right
\|_{L^2(\tau,T,W^{-1,r'} (\Omega))} &\le& C  ,
\end{eqnarray}
for all $r' \in (2,+\infty)$ where $C$ does not depend on $\tau>0$
(but does depend on $r'$).
\end{lemma}
%-----------------------------------------
\begin{proof}
  Estimate~\eqref{estim:approx1} and \eqref{estim:approx3} are direct
  consequences of \eqref{estim:conservation mass tau} and
  \eqref{estim:nrj tau}.

Next, we note that 
$$
\frac{u^\tau - S_\tau u^\tau}\tau =  \partial_x \bigg(-f_\eps
(u^\tau)\partial_x I (u^\tau)  \bigg) \, .
$$
The bound \eqref{estim:approx1} and Sobolev embedding theorems imply
that the sequence $(u^\tau)_{\tau>0}$ is bounded in $L^\infty(0,T,L^{p}(\Omega))$ for
all $p<\infty$ and so $ f_\eps(u^\tau)$ is bounded in
$L^\infty(0,T,L^{p}(\Omega))$ for all $p<\infty$.  Since $ \partial_x
I (u^\tau)$ is bounded in $L^2(Q)$, we deduce that
$f_\eps(u^\tau) \partial_x I (u^\tau)$ is bounded in $L^2
(\tau,T,L^r(\Omega))$ for all $r\in [1,2)$.  It follows that
$$
\pa_x \left( f_\eps(u^\tau) \partial_x I (u^\tau )\right) \text{ is
  bounded in } L^2(\tau,T,W^{-1,r'}(\Omega)) 
$$
for all $r' \in (2,\infty)$. 
\end{proof}

\paragraph{Step 2: Compactness result.}
Thanks to the following imbeddings
$$
\ha \hookrightarrow L^q(\Omega) \rightarrow W^{-1,r'}(\Omega)
$$
for all $q<\infty$,
we can use
Aubin's lemma to obtain that $(u^\tau)_\tau$ is relatively compact in
$\mathcal{C}^0 (0,T,L^q(\Omega))$ for all $q<\infty$. 

Remark that $(\partial_x I(u^\tau))_\tau$ is bounded in $L^2 (Q)$ and $(u^\tau)_\tau$ is
bounded in $L^\infty(0,T;L^1 (\Omega))$. It follows that $(u^\tau)_\tau$ is
bounded in $ L^2 (0,T,H^2_N(\Omega))$. Since
$$
H^2_N(\Omega) \hookrightarrow H^{\frac{3}{2}}_N(\Omega) \rightarrow W^{-1,r'}(\Omega)
$$
we deduce that $(u^\tau)_\tau$ is relatively compact in
$L^2(0,T;H^{\frac{3}{2}}_N(\Omega))$.  Up to a subsequence, we can
thus assume that as $\tau\rightarrow 0$, we have the following convergences:
\begin{itemize}
\item $u^\tau \to u^\eps\in L^\infty(0,T,\ha) $ almost everywhere in $Q$;
\item $u^\tau \to u^\eps $ in $L^2(0,T,H^1(\Omega))$ strong;
\item $\partial_x I (u^\tau) \rightharpoonup \partial_x I (u^\eps)$ in
  $L^2 (Q)$ weak.
\end{itemize} 

\paragraph{Step 3: Derivation of Equation \eqref{eq:eps}.}
We want to pass to the limit in \eqref{eq:approx discrete}. 

We fix $\varphi \in \mathcal{C}_c^1 ([0,T),H^1(\Omega))$. Then
\begin{multline*}
  \iint_{Q_\tau} \frac{u^\tau - S_\tau u^\tau}{\tau} \varphi =
  \iint_{Q} u^\tau (x,t) \frac{\varphi (x,t) - \varphi
    (t+\tau,x)}\tau \\
  - \frac1\tau \int_0^\tau \intom u^\tau(x,t) \varphi (x,t) \, dx +
  \frac1\tau \int_{T-\tau}^T \intom u^\tau(x,t) \varphi (x,t) \, dx .
\end{multline*}
We deduce:
$$
\iint_{Q_\tau} \frac{u^\tau  - S_\tau u^\tau}{\tau} \varphi  
\to  -\iint_Q u^\eps (\partial_t \varphi) 
 -\intom u^\eps (0,x)  \vphi (0,x) \, dx + 0 .
$$
It remains to pass to the limit in the non-linear term. Let $\eta>0$.
Since $u^\tau \to u^\eps$ almost everywhere in $Q$, Egorov's theorem
yields the existence of a set $A_\eta \subset Q$ such that $|Q
\setminus A_\eta |\le \eta$ and
$$
u^\tau \to u^\eps \text{ uniformly in } A_\eta \, .
$$
In particular, 
$$
\sqrt{f_\eps (u^\tau)} \partial_x  \vphi \to \sqrt{f_\eps(u^\eps)} \partial_x  \vphi \text{ in }
L^2 (A_\eta) 
$$
and
\begin{equation}\label{conv-nonlinear-tau}
\sqrt{f_\eps (u^\tau)} \partial_x I (u^\tau) \rightharpoonup \sqrt{f_\eps
(u^\eps)} \partial_x I (u^\eps) \text{ in } L^2 (A_\eta) \, .
\end{equation}
Hence
$$
\int_{A_\eta} f_\eps (u^\tau) \partial_x I (u^\tau) \partial_x  \vphi \to 
\int_{A_\eta} f_\eps (u^\eps) \partial_x I (u^\eps) \partial_x  \vphi
$$
as $\tau$ goes to zero. 

Finally,  we look at what happens on $Q \setminus A_\eta$. Choose 
$p_1,p_2, p_3$ such that $\sum_i p_i^{-1}=1$ and write
\begin{multline*}
  \int_{Q \setminus A_\eta} |f_\eps(u^\tau) \partial_x I
 ( u^\tau) \partial_x  \vphi| \\ \le \|\partial_x 
  \vphi\|_{L^\infty(0,T,L^{p_1} (\Omega))} \int_0^T \| f_\eps
  (u^\tau) \partial_x I( u^\tau) \|_{L^{p_2} (\Omega)}
\|  \un_{Q \setminus A_\eta} \|_{L^{p_3} (\Omega)}   \\
  \le \|\partial_x \vphi\|_{L^\infty(0,T,L^{p_1} (\Omega))} \| f_\eps
  (u^\tau) \partial_x I (u^\tau) \|_{L^2(0,T,L^{p_2} (\Omega))} \|
  \un_{Q \setminus A_\eta} \|_{L^2(0,T,L^{p_3} (\Omega))} \, .
\end{multline*}
We now choose $p_2 \in [1,2)$ (and so $p_1>2$ and $p_3 >2$). 
$$
  \int_{Q \setminus A_\eta} |f_\eps(u^\tau) \partial_x I
  (u^\tau) \partial_x \vphi| \le C(\vphi) \|  \un_{Q \setminus A_\eta}
  \|_{L^{p_3} (Q)} \le C(\vphi) \eta^{\frac1{p_3}} \, .
$$
Since $\eta$ is arbitrary, the proof is complete. 

\paragraph{Step 4: Inequalities.}
Since $u^\tau \to u^\eps$ in $L^\infty(0,T,L^1(\Omega))$, mass
conservation equation~\eqref{estim:conservation mass eps} follows from
\eqref{estim:conservation mass tau}.  \smallskip

Estimate~\eqref{estim:nrj eps} follows from \eqref{estim:nrj
  tau}. Indeed, since $u^\tau \to u^\eps$ almost everywhere,
Proposition~\ref{prop:semi-norms} and Fatou's lemma imply that for
almost $t \in (0,T)$
$$
\|u^\eps (t)\|^2_{\dot{H}^{\frac12}(\Omega)} \le \liminf_{\tau \to 0}
\|u^\tau (t)\|^2_{\dot{H}^{\frac12}(\Omega)}.
$$
Thanks to \eqref{conv-nonlinear-tau}, we also have
\begin{multline*}
\int_0^T\int_\Omega f_\eps (u^\eps) |\partial_x I (u^\eps)|^2 \, dx\, dt\leq
\liminf_{\tau\to 0} \int_0^T\int_\Omega f_\eps (u^\tau) (\partial_x I
(u^\tau))^2\, dx\,dt \\ +\int_0^T\int_{\Omega \setminus A_\eta}
 f_\eps (u^\eps) |\partial_x I (u^\eps)|^2 \, dx\, dt.
\end{multline*}
Letting $\eta \to 0$ permits to conclude.
\smallskip

To derive \eqref{estim:entropy eps} we note that 
 $G_\eps(u^\tau)\to F_\eps(u^\eps)$ almost everywhere. So
Fatou's Lemma implies for almost every $t \in (0,T)$
$$ 
\int_\Omega G_\eps(u^\eps(x,t))\, dx \leq \liminf_{\tau \to 0}
\int_\Omega G_\eps(u^\tau(x,t))\, dx \leq \int_\Omega G_\eps(u_0)\,
dx.
$$
Finally, since $(u^\tau)_\tau$ is relatively compact in
$L^2(0,T;H^{\frac{3}{2}}_N(\Omega))$, we have
$$ 
\int_0^t ||u^\eps(s) ||^2_{\dot{H}^{\frac32}}\, ds = \lim_{\tau\to 0}
\int_0^t ||u^\tau (s) ||^2_{\dot{H}^{\frac32}}\, ds
$$
and so \eqref{estim:entropy eps}  follows from \eqref{estim:entropy tau}.

\nocite{*}
\bibliographystyle{siam}
\bibliography{crack}

\end{document}